\documentclass[10pt]{amsart}
\usepackage{amsmath}
\usepackage{amssymb}
\usepackage{amsfonts}
\usepackage{mathrsfs}
\usepackage{enumerate}
\usepackage{amscd}
\usepackage{graphicx}
\usepackage{epsfig}
\usepackage[active]{srcltx}
\numberwithin{equation}{section}       

\theoremstyle{plain}

\newtheorem{prop}{Proposition}[section]

\newtheorem{coro}[prop]{Corollary}
\newtheorem{lemm}[prop]{Lemma}

\newtheorem{theoalph}{Theorem}

\newtheorem{propalph}[theoalph]{Proposition}

\theoremstyle{definition}
\newtheorem{defi}[prop]{Definition}

\theoremstyle{remark}

\newtheoremstyle{citing}
  {3pt}
  {3pt}
  {\itshape}
  {}
  {\bfseries}
  {.}
  {.5em}
  {\thmnote{#3}}

\theoremstyle{citing}
\newtheorem*{generic}{}

\newcommand{\partn}[1]{{\smallskip \noindent \textbf{#1.}}}

\DeclareMathAlphabet{\mathpzc}{OT1}{pzc}{m}{it} 

\newcommand{\C}{\mathbb{C}}
\newcommand{\D}{\mathbb{D}}

\newcommand{\N}{\mathbb{N}}

\newcommand{\R}{\mathbb{R}}

\newcommand{\Z}{\mathbb{Z}}

\newcommand{\cE}{\mathcal{E}}

\newcommand{\cI}{\mathcal{I}}

\newcommand{\cK}{\mathcal{K}}

\newcommand{\cM}{\mathcal{M}}

\newcommand{\cP}{\mathcal{P}}

\newcommand{\cR}{\mathcal{R}}

\newcommand{\cW}{\mathcal{W}}
\newcommand{\cX}{\mathcal{X}}

\newcommand{\fD}{\mathfrak{D}}

\newcommand{\sM}{\mathscr{M}}

\newcommand{\sP}{\mathscr{P}}

\newcommand{\hC}{\widehat{C}}

\newcommand{\hU}{\widehat{U}}
\newcommand{\hV}{\widehat{V}}
\newcommand{\hW}{\widehat{W}}

\newcommand{\tB}{\widetilde{B}}

\newcommand{\tU}{\widetilde{U}}

\newcommand{\tW}{\widetilde{W}}

\newcommand{\tY}{\widetilde{Y}}

\newcommand{\talpha}{\widetilde{\alpha}}

\newcommand{\tzeta}{\widetilde{\zeta}}
\newcommand{\teta}{\widetilde{\teta}}

\newcommand{\tpsi}{\widetilde{\psi}}

\renewcommand{\=}{ : = }

\DeclareMathOperator{\diam}{diam}

\DeclareMathOperator{\dist}{dist}

\newcommand{\cl}[1]{{\rm cl}({#1})}

\DeclareMathOperator{\crit}{crit}
\newcommand{\wtp}{\widetilde{p}}

\newcommand{\whc}{\widehat{c}}
\newcommand{\chicrit}{\chi_{\crit}(c)}

\begin{document}

\title[High-order phase transitions in the quadratic family]{High-order phase transitions in the \\ quadratic family}
\author{Daniel Coronel}
\address{Daniel Coronel, Departamento de Matem{\'a}ticas, Universidad Andr{\'e}s Bello, Avenida Rep{\'u}blica~220, segundo piso, Santiago, Chile}
\email{alvaro.coronel@unab.cl}
\author{Juan Rivera-Letelier}
\address{Juan Rivera-Letelier, Facultad de Matem{\'a}ticas, Pontifica Universidad Cat{\'o}lica de Chile, Avenida Vicu{\~n}a Mackenna~4860, Santiago, Chile}
\email{riveraletelier@mat.puc.cl}

\begin{abstract}
We give the first example of a transitive quadratic map whose real and complex geometric pressure functions have a high-order phase transition.
In fact, near the phase transition these functions behave as
$$ x \mapsto \exp \left( - x^{-2} \right) $$
near~$x = 0$, before becoming linear.
This quadratic map has a non-recurrent critical point, so it is non-uniformly hyperbolic in a strong sense.
\end{abstract}

\maketitle

%
%

\section{Introduction}
This paper is concerned with the thermodynamic formalism of smooth dynamical systems.
Such a study was initiated by Sinai, Ruelle, and Bowen \cite{Sin72,Bow75,Rue76} in the context of uniformly hyperbolic diffeomorphisms and H{\"o}lder continuous potentials.
In the last decades there has been important efforts to extend the theory beyond the uniformly hyperbolic setting, specially in real and complex dimension~$1$ where a complete picture is emerging, see for example~\cite{BruTod09, MakSmi00, MakSmi03, PesSen08, PrzRiv11, PrzRivinterval} and references therein.
See also~\cite{Sar11, UrbZdu0901, VarVia10} and references therein for (recent) results in higher dimensions.

For a smooth map~$f$ in real or complex dimension~$1$ and a real parameter~$t$, we consider the pressure of~$f$ with respect to the geometric potential~$- t \log |Df|$, see~\S\ref{ss:statements} for precisions.
The function of~$t$ so defined is the \emph{geometric pressure function of~$f$}.
It is closely related to several multifractal spectra and large deviation rate functions associated with~$f$, see for example~\cite[Lemma~$2$]{BinMakSmi03}, \cite{GelPrzRam10}, \cite{IomTod11}, \cite[Theorems~$1.2$ and~$1.3$]{KelNow92}, \cite[Appendix~B]{PrzRiv11}, and references therein.

We exhibit a transitive quadratic map whose geometric pressure function behaves, for some constants~$A > 0$ and~$\chi > 0$ and for~$t$ near a certain parameter~$t_*$, as the function
$$ t \mapsto
\begin{cases}
- t \chi + \exp( - A (t_* - t)^{-2}) & \text{if } t < t_*; \\
- t \chi & \text{if } t \ge t_*,
\end{cases}
$$
see the Main Theorem in~\S\ref{ss:statements}.
In particular, the geometric pressure function of this map is not real analytic at~$t = t_*$; that is, it has a \emph{phase transition} at~$t = t_*$ in the sense of statistical mechanics.
This is the first example of a transitive smooth dynamical system having a phase transition of infinite contact order.
This example is also robust: Every family of sufficiently regular unimodal maps that is close to the quadratic family has a member with the same property.

The quadratic map we study has a non-recurrent critical point, so it is non-uniformly hyperbolic in a strong sense.
Thus, roughly speaking, lack of expansion is not responsible for the phase transition.
Instead, it is the irregular behavior of the critical orbit that is one of the mechanisms behind the phase transition.
Considering a different behavior of the critical orbit, in the companion paper~\cite{CorRiva} we gave the first example of a quadratic map having a phase transition at a large value of~$t$; that is, of a ``low-temperature phase transition'', see also~\cite[\S$5$]{MakSmi03} for some conformal Cantor sets with similar properties.
In contrast with the example studied here, the geometric pressure function of the quadratic map studied in~\cite{CorRiva} is not differentiable at the phase transition; that is, it is a phase transition of \emph{first order}.

Another interesting feature of the quadratic map we study is that it has no equilibrium state at the phase transition.
At a low-temperature phase transition there can be at most~$1$ equilibrium state,\footnote{See~\cite[Theorem~$6$]{Dob1304} in the real setting, and~\cite[Theorem~$8$]{Dob12} in the complex setting.} and in the companion paper~\cite{CorRiva} we provide an example of a quadratic map having one.

There are various examples in the literature of transitive smooth maps whose geometric pressure function has a first-order phase transition.
This includes quadratic maps that have an absolutely continuous invariant measure, and that do not satisfy the Collet-Eckmann condition.\footnote{For such a map, the geometric pressure function is identically zero after its first zero, see~\cite[Theorem~A]{NowSan98} or~\cite[Corollary~$1.3$]{Riv1204} in the real case and \cite[Main Theorem]{PrzRivSmi03} in the complex case.
On the other hand, since every absolutely continuous invariant measure has a strictly positive Lyapunov exponent, the existence of such a measure easily implies that the geometric pressure function is not differentiable at its first zero.}
By the work of Makarov and Smirnov~\cite{MakSmi00}, this also includes those phase transitions in the complex case that occur at a parameter in~$(- \infty, 0)$.
See also~\cite{DiaGelRam11, DiaGelRam1303, LepOliRio11} for examples of first-order phase transitions of some transitive $3$-dimensional diffeomorphisms.

Another type of phase transition that has been studied in detail, is that related to the existence of a neutral periodic point.
For a given~$\alpha \ge 1$, Lopes shows in~\cite{Lop93} that for~$t < 1$ close to~$1$ the geometric pressure function of the map~$f_{\alpha}$ given by~$x \mapsto x(1 + x^\alpha) \mod 1$, is of the order of~$(1 - t)^{\alpha}$; on the other hand, this function is identically zero on~$[1, + \infty)$.
In view of this result, it is expected that for a quadratic map~$f$ having a periodic point~$p$ of period~$n \ge 1$ satisfying~$Df^n(p) = \pm 1$, the geometric pressure function of~$f$ has a unique phase transition, and that this phase transition is of finite order.\footnote{Notice that, when~$Df^n(p) = 1$ (resp.~$Df^n(p) = -1$), the function~$x \mapsto f^n(x) - x$ (resp. $x \mapsto f^{2n}(x) - x$) is of the order of~$(x - p)^2$ (resp. $(x - p)^3$) near~$p$, see for example~\cite{CarGam93,Mil06}.}

\subsection{Statement of results}
\label{ss:statements}
We consider a set of real parameters~$c$ close to~$-2$, such that~$f_c(c) > c$, such that the interval~$I_c \= [c, f_c(c)]$ of~$\R$ is invariant by~$f_c$, and such that~$f_c$ is topologically exact on this set.
We consider~$2$ dynamical systems associated to~$f_c$: The interval map~$f_c|_{I_c}$, and the complex quadratic polynomial~$f_c$ acting on its Julia set~$J_c$.

For such~$c$, define
$$ \chicrit \= \liminf_{m \to + \infty} \frac{1}{m} \log |Df_c^m(c)|, $$
and denote by~$\sM_c^{\R}$ the space of Borel probability measures supported on~$I_c$ that are invariant by~$f_c$.
For a measure~$\mu$ in~$\sM_c^{\R}$ denote by~$h_\mu(f_c)$ the measure-theoretic entropy of~$f_c$ with respect to~$\mu$ and for each~$t$ in~$\R$ put
$$ P_c^{\R}(t)
\=
\sup \left\{ h_\mu(f_c) - t \int \log |Df_c| d\mu \mid \mu \in \sM_c^{\R} \right\}, $$
which is finite.
The function~$P_c^{\R} : \R \to \R$ so defined is called the \emph{geometric pressure function of~$f_c|_{I_c}$}; it is convex and nonincreasing.

Similarly, denote by~$\sM_c^{\C}$ the space of Borel probability measures supported on~$J_c$ that are invariant by~$f_c$ and for a measure~$\mu$ in~$\sM_c^{\C}$ we denote by~$h_\mu(f_c)$ the measure-theoretic entropy of~$f_c$ with respect to~$\mu$.
Then the \emph{geometric pressure function $P_c^{\C} : \R \to \R$
of~$f_c$} is defined by
$$ P_c^{\C}(t)
\=
\sup \left\{ h_\mu(f_c) - t \int \log |Df_c| d\mu \mid \mu \in \sM_c^{\C}
\right\}. $$

\begin{generic}[Main Theorem]
There is a real parameter~$c$ such that the critical point of~$f_c$ is non-recurrent, such that for some~$t_* > 0$ and every~$t \ge t_*$, we have
$$
P_c^\R(t) =  P_c^{\C}(t) = - t \frac{\chicrit}{2},
$$
and such that for some constants~$A > 0$, $B^+ > 0$, and $B^- > 0$, we have for every~$t$ in~$(0, t_*)$ close to~$t_*$
\begin{multline*}
- t \frac{\chicrit}{2} +  2^{-\left(\frac{A}{t_* - t} + B^-\right)^2}
\le
P_c^\R(t)
\le
P_c^{\C}(t)
\\ \le 
- t \frac{\chicrit}{2} + 2^{-\left(\frac{A}{t_* - t} - B^+ \right)^2}.
\end{multline*}
In particular, both~$P_c^{\R}$ and~$P_c^{\C}$ are of class~$C^2$ at~$t = t_*$, but neither of these functions is real analytic at~$t = t_*$.
\end{generic}

We show in addition that there is no equilibrium state at the phase transition, that there is a unique associated conformal measure, and that this last measure is dissipative and purely atomic, see~\S\ref{ss:main technical theorem} for definitions and for a strengthened version of the Main Theorem.
It can also be shown that, if for each~$t$ in~$(0, t_*)$ we denote by~$\nu_t$ the unique equilibrium state of~$f_c$ for the potential~$- t \log |Df_c|$, then the measure~$\nu_t$ converges as~$t \mapsto t_*^-$ to the invariant probability measure supported on a certain periodic point of period~$3$ of~$f_c$.

Since the critical point of a map~$f_c$ as in the Main Theorem is non-recurrent, it follows that~$f_c$ satisfies the Collet-Eckmann condition: $\chicrit > 0$, see~\cite{Mis81} for the real case and~\cite{Man93} for the complex case.
So, $t_*$ in the Main Theorem is strictly larger than the first zero of the geometric pressure function of~$f_c$; that is, $f_c$ has a ``low-temperature'' phase transition at~$t = t_*$ in the sense of~\cite{CorRiva}.

\subsection{Notes and references}
\label{ss:notes and references}
For complex rational maps, Makarov and Smirnov showed that every phase transition occurring at a negative parameter is removable, in the sense that the geometric pressure function has a real analytic continuation to all of~$(- \infty, 0)$, see~\cite[Theorem~B]{MakSmi00}.
In contrast, the geometric pressure function of a map as in the Main Theorem cannot admit a real analytic continuation beyond the phase transition.

For a map as in the Main Theorem, the non-existence of equilibrium states also follows from~\cite[Corollary~$1.3$]{InoRiv12}.

For a quadratic map having a phase transition at the first zero of the pressure function, that is, a \emph{high-temperature phase transition}, the number of ergodic equilibrium states can be arbitrary, see~\cite[Corollaries~$2$ and~$3$]{CorRiv10b}, and also~\cite[Example~$5.4$]{BruKel98} and~\cite[Corollary~$2$]{BruTod06} for an example having no equilibrium state.

Bruin and Todd study in~\cite{BruTod1202} certain piecewise linear models (with an infinite number of break points) of the smooth unimodal maps having a wild attractor in~\cite{BruKelNowvSt96}.
They show that for a large value of the order of the critical point, the piecewise linear model has a high-order phase transition.
Notice that no quadratic map can have a wild attractor, see~\cite{Lyu94b}.
\subsection{Strategy and organization}
\label{ss:organization}
To prove the Main Theorem, we consider the set of parameters introduced in~\cite{CorRiva}.
For each parameter~$c$ in this set, the critical value is eventually mapped to an expanding Cantor set, denoted by~$\Lambda_c$.
For such a parameter, the behavior of the geometric pressure function at low temperatures is intimately related to the derivatives of the map along the critical orbit (Proposition~\ref{p:improved MS criterion}).
As a first approximation we use the multipliers of the~$2$ periodic orbits of period~$3$ of~$f_c$ to estimate these derivatives.
However, the distortion constants in these estimates are too big to achieve the level of precision needed to prove the Main Theorem.
To achieve a higher precision, we estimate these distortion constants in terms of the total distortion along certain homoclinic orbits connecting the~$2$ periodic orbits of period~$3$ (Proposition~\ref{p:Improved distortion estimate} in~\S\ref{ss:Improved distortion estimate}).

We now proceed to describe the organization of the paper more precisely.

After some preliminaries in~\S\ref{s:preliminaries}, we state an strengthened
version of the Main Theorem in~\S\ref{ss:main technical theorem}, as the ``Main Technical Theorem''.
In~\S\ref{ss:2 variables series} we introduce an abstract~$2$ variables series that captures the behavior of the geometric pressure function at low temperatures (Proposition~\ref{p:2 variables series}).
Its definition is based on an approximation of the derivatives at the critical value in terms of its itinerary in~$\Lambda_c$ (Proposition~\ref{p:Improved distortion estimate} in~\S\ref{ss:Improved distortion estimate}), as mentioned above.

In~\S\ref{s:reduction}, which is independent of the rest of the paper, we study in an abstract setting the~$2$ variables series for an specific class of itineraries. 
We show that this series has a phase transition with an asymptotic behavior as in the Main Theorem.
The itineraries are defined in~\S\ref{ss:the itinerary}, and the estimates of the corresponding~$2$ variables series are made in~\S\ref{ss:estimating 2 variables series}.

The proof of the Main Technical Theorem is given in~\S\ref{s:finale}.
After some general results about conformal measures in~\S\ref{ss:conformal measures}, we make some technical estimates in~\S\ref{ss:transition parameter}.
The proof of the Main Technical Theorem is in~\S\ref{ss:proof of Main Technical
Theorem}, 
after recalling a few results from~\cite{CorRiva}.

\subsection{Acknowledgments}
The first named author acknowledges partial support from FONDECYT grant 11121453.
This article was completed while second named author was visiting Brown University and the Institute for Computational and Experimental Research in Mathematics (ICERM).
He thanks both of these institutions for the optimal working conditions provided, and acknowledges partial support from FONDECYT grant 1100922.

\section{Preliminaries}
\label{s:preliminaries}
We use~$\N$ to denote the set of integers that are greater than or equal to~$1$ and~$\N_0 \= \N \cup \{ 0 \}$.

\subsection{Quadratic polynomials, Green functions, and B{\"o}ttcher coordinates}
\label{ss:quadratic polynomials}
In this subsection and the next we recall some basic facts about the dynamics of
complex quadratic polynomials, see for instance~\cite{CarGam93} or~\cite{Mil06}
for references.

For~$c$ in~$\C$ we denote by~$f_c$ the complex quadratic polynomial
$$ f_c(z) = z^2 + c, $$
and by~$K_c$ the \emph{filled Julia set} of $f_c$; that is, the set of all
points~$z$ in~$\C$ whose forward orbit under~$f_c$ is bounded in~$\C$.
The set~$K_c$ is compact and its complement is the connected set consisting of
all points whose orbit converges to infinity in the Riemann sphere.
Furthermore, we have $f_c^{-1}(K_c) = K_c$ and~$f_c(K_c) = K_c$. 
The boundary~$J_c$ of~$K_c$ is the \emph{Julia set of~$f_c$}.

For a parameter~$c$ in~$\C$, the \emph{Green function of~$K_c$} is the function
$G_c:\C \to [0,+\infty)$ that is identically~$0$ on~$K_c$, and that
for~$z$ outside~$K_c$ is given by the limit,
\begin{equation}\label{def:Green function}
  G_c(z) = \lim_{n\rightarrow +\infty} \frac{1}{2^n} \log |f_c^n(z)| > 0.
\end{equation}
The function~$G_c$ is continuous, subharmonic, satisfies~$G_c \circ f_c = 2G_c$ on~$\C$, and it is harmonic and strictly positive outside~$K_c$.
On the other hand, the critical values of~$G_c$ are bounded from above by~$G_c(0)$, and
the open set
$$ U_c \= \{z\in \C \mid G_c(z) > G_c(0)\} $$
is homeomorphic to a punctured disk.
Notice that $G_c(c)=2G_c(0)$, thus~$U_c$ contains~$c$. 

By B{\"o}ttcher's Theorem there is a unique conformal representation
\[
 \varphi_c: U_c
\rightarrow
\{z\in \C \mid |z| > \exp (G_c(0)) \},
\]
and this map conjugates~$f_c$ to $z \mapsto z^2$.
It is called \emph{the B{\"o}ttcher coordinate of~$f_c$} and satisfies $G_c =
\log |\varphi_c|$.

\subsection{External rays and equipotentials}
\label{ss:rays and equipotentials}
Let~$c$ be in~$\C$.
For~$v > 0$ the \emph{equipotential~$v$ of~$f_c$} is by definition~$G_c^{-1}(v)$.
A \emph{Green's line of~$G_c$} is a smooth curve on the complement of~$K_c$ in~$\C$ that is orthogonal to the equipotentials of~$G_c$ and that is maximal with this property. 
Given~$t$ in~$\R / \Z$, the \emph{external ray of angle~$t$ of~$f_c$}, denoted by~$R_c(t)$, is the Green's line of~$G_c$ containing
$$ \{ \varphi_c^{-1}(r \exp(2 \pi i t)) \mid \exp(G_c(0))< r < +\infty \}. $$
By the identity~$G_c \circ f_c= 2G_c$, for each~$v > 0$ and each~$t$ in~$\R / \Z$ the map~$f_c$ maps the equipotential~$v$ to the equipotential~$2v$ and maps~$R_c(t)$ to~$R_c(2t)$.
For~$t$ in~$\R / \Z$ the external ray~$R_c(t)$ \emph{lands at a point~$z$}, if~$G_c : R_c(t) \to (0, + \infty)$ is a bijection and if~$G_c|_{R_c(t)}^{-1}(v)$ converges to~$z$ as~$v$ converges to~$0$ in~$(0, + \infty)$.
By the continuity of~$G_c$, every landing point is in $J_c = \partial K_c$.

The \emph{Mandelbrot set~$\cM$} is the subset of~$\C$ of those
parameters~$c$ for which~$K_c$ is connected.
The function
\[
 \begin{array}{cccl}
  \Phi : &\C \setminus \cM & \to & \C \setminus \cl{\D}\\
          &     c         &  \mapsto           & \Phi(c) \= \varphi_c(c)
 \end{array}
\]
is a conformal representation, see~\cite[VIII, \emph{Th{\'e}or{\`e}me}~$1$]{DouHub84}.
For~$v > 0$ the \emph{equipotential~$v$ of~$\cM$} is by definition
$$ \cE(v) \= \Phi^{-1}(\{z\in \C \mid |z| = v \}). $$ 
On the other hand, for~$t$ in~$\R / \Z$ the set
$$ \cR(t) \= \Phi^{-1}(\{r \exp(2 \pi i t) \mid r > 1 \}) $$
is called the \emph{external ray of angle~$t$ of~$\cM$}.
We say that $\cR(t)$ \emph{lands at a point~$z$} in~$\C$, if~$\Phi^{-1} (r \exp(2\pi i t))$ converges to~$z$ as $r \searrow 1$.
When this happens~$z$ belongs to~$\partial \cM$.

\subsection{The wake~$1/2$}
\label{ss:wake 1/2}
In this subsection we recall a few facts that can be found for example
in~\cite{DouHub84} or~\cite{Mil00c}.

The external rays~$\cR(1/3)$ and~$\cR(2/3)$ of~$\cM$ land at the parameter~$c = -3/4$, and these are the only external rays of~$\cM$ that land at this point, see
for example~\cite[Theorem~$1.2$]{Mil00c}.
In particular, the complement in~$\C$ of the set
$$ \cR(1/3) \cup \cR(2/3) \cup \{ - 3/4 \} $$
has~$2$ connected components; we denote by~$\cW$ the connected component
containing the point~$c = -2$ of~$\cM$.

For each parameter~$c$ in~$\cW$ the map~$f_c$ has~$2$ distinct fixed points; one
of the them is the landing point of the external ray~$R_c(0)$ and it is denoted
by~$\beta(c)$; the other one is denoted by~$\alpha(c)$.
The only external ray landing at~$\beta(c)$ is~$R_c(0)$, and
the only external ray landing at~$-\beta(c)$
is~$R_c(1/2)$.

Moreover, for every parameter~$c$ in~$\cW$ the only external rays 
of~$f_c$ landing at~$\alpha(c)$ are~$R_c(1/3)$
and~$R_c(2/3)$, see for example~\cite[Theorem~$1.2$]{Mil00c}.
The complement of~$R_c(1/3) \cup R_c(2/3) \cup \{ \alpha(c)
\}$ in~$\C$ has~$2$ connected components; one containing~$- \beta(c)$ and~$z =
c$, and the other one containing~$\beta(c)$ and~$z = 0$.
On the other hand, the point~$\alpha(c)$ has~$2$ preimages by~$f_c$: Itself and~$\talpha(c) \= - \alpha(c)$.
The only external rays landing at~$\talpha(c)$ are~$R_c(1/6)$
and~$R_c(5/6)$.

\subsection{Yoccoz puzzles and para-puzzle}
\label{ss:puzzles}
In this subsection we recall the definitions of Yoccoz puzzle and para-puzzle.
We follow \cite{Roe00}.

\begin{defi}[Yoccoz puzzles]
Fix~$c$ in~$\cW$ and consider the open region $X_c \= \{z\in \C \mid G_c(z) <
1\}$. 
The \emph{Yoccoz puzzle of~$f_c$} is given by the following sequence of
graphs~$(I_{c, n})_{n = 0}^{+ \infty}$ defined for~$n = 0$ by:
\[
 I_{c,0} \= \partial X_c \cup (X_c \cap \cl{R_c(1/3)} \cap \cl{R_c(2/3)}),
\]
and for~$n \ge 1$ by~$I_{c,n} \= f_c^{-n}(I_{c,0})$.
The \emph{puzzle pieces of depth~$n$} are the connected components of $f_c^{-n}(X_c) \setminus I_{c,n}$.
The puzzle piece of depth~$n$ containing a point~$z$ is denoted by~$P_{c,n}(z)$.
\end{defi}

Note that for a real parameter~$c$, every puzzle piece intersecting the real line is invariant under complex conjugation.
Since puzzle pieces are simply-connected, it follows that the intersection of such a puzzle piece with~$\R$ is an interval.

\begin{defi}[Yoccoz para-puzzles\footnote{In contrast with~\cite{Roe00}, we only consider para-puzzles contained in~$\cW$.}]
Given an integer~$n \ge 0$, put
$$ J_n
\=
\{t\in [1/3,2/3] \mid 2^n t ~ (\mathrm{mod}\, 1) \in \{1/3,2/3\} \}, $$
let~$\cX_n$ be the intersection of~$\cW$ with the open region in the parameter plane bounded by the equipotential~$\cE(2^{-n})$ of~$\cM$, and put
\[
 \cI_{n}
\=
\partial \cX_n \cup \left( \cX_n \cap \bigcup_{t\in J_n} \cl{\cR(t)} \right).
\]
Then the \emph{Yoccoz para-puzzle of~$\cW$} is the sequence of graphs~$(\cI_n)_{n = 0}^{+ \infty}$.
The \emph{para-puzzle pieces of depth~$n$} are the connected components of $\cX_n \setminus \cI_n$.
The para-puzzle piece of depth~$n$ containing a parameter~$c$ is denoted by~$\cP_n(c)$.  
\end{defi}
Observe that there is only~$1$ para-puzzle piece of depth~$0$, and only~$1$ para-puzzle piece of depth~$1$; they are bounded by the same external rays but different equipotentials.
Both of them contain~$c = - 2$.

Fix a parameter~$c$ in~$\cP_0(-2)$.
There are precisely~$2$ puzzle pieces of depth~$0$: $P_{c, 0}(\beta(c))$ and~$P_{c, 0}(-\beta(c))$.
Each of them is bounded by the equipotential~$1$ and by the closures of the
external rays landing at~$\alpha(c)$.
Furthermore, the critical value~$c$ of~$f_c$ is contained 
in~$P_{c, 0}(- \beta(c))$ and the critical point in~$P_{c, 0}(\beta(c))$.
It follows that the set~$f_c^{-1}(P_{c, 0}(\beta(c)))$ is the 
disjoint union of~$P_{c, 1}(- \beta(c))$ and~$P_{c, 1}(\beta(c))$, so~$f_c$ maps each of the sets~$P_{c, 1}(-\beta(c))$ and~$P_{c,
1}(\beta(c))$ biholomorphically to~$P_{c, 0}(\beta(c))$.
Moreover, there are precisely~$3$ puzzle pieces of depth~$1$: 
$$ P_{c, 1}(-\beta(c)),
P_{c, 1}(0)
\text{ and }
P_{c, 1}(\beta(c)); $$
$P_{c, 1}(- \beta(c))$ is bounded by the equipotential~$1/2$ and by the 
closures of the external rays that land at~$\alpha(c)$; $P_{c, 1}(\beta(c))$ is
bounded by the equipotential~$1/2$ and by the closures of the external rays that
land at~$\talpha(c)$; and~$P_{c, 1}(0)$ is bounded by the equipotential~$1/2$
and by the closures of the external rays that land at~$\alpha(c)$ and
at~$\talpha(c)$.
In particular, the closure of~$P_{c, 1}(\beta(c))$ is 
contained in~$P_{c, 0}(\beta(c))$.
It follows from this that for each integer~$n \ge 1$ the map~$f_c^n$ maps~$P_{c, n}(- \beta(c))$ biholomorphically to~$P_{c, 0}(\beta(c))$.

The following is used several times, see~\cite[Lemma~$3.3$]{CorRiva}.
\begin{lemm}
\label{l:auxiliary para-puzzle pieces}
For each integer~$n \ge 1$, the following properties hold.
\begin{enumerate}[1.]
\item 
The para-puzzle piece~$\cP_n(-2)$ contains the closure of~$\cP_{n + 1}(-2)$.
\item 
For each  parameter~$c$ in~$\cP_n(-2)$  the 
critical value~$c$ of~$f_c$ is in~$P_{c, n}(-\beta(c))$.
\end{enumerate}
\end{lemm}

\subsection{The uniformly expanding Cantor set}
\label{ss:expanding Cantor set}
For a parameter~$c$ in~$\cP_3(-2)$, the maximal invariant set~$\Lambda_c$ of~$f_c^3$ in~$P_{c, 1}(0)$ plays an important role in the proof of the Main Theorem.
After recalling some of the properties of~$\Lambda_c$ shown 
in~\cite[\S$3.3$]{CorRiva}, in this subsection we prove that~$f_c^3$ is
uniformly expanding on~$\Lambda_c$ and make some distortion estimates
for~$f_c^3$ on~$\Lambda_c$ (Lemma~\ref{l:contractions}).

Fix~$c$ in~$\cP_3(-2)$.
There are precisely~$2$ connected components of~$f_c^{-3}(P_{c, 1}(0))$ contained in~$P_{c, 1}(0)$ that we denote by~$Y_c$ and~$\tY_c$.
The closures of these sets are disjoint and contained in~$P_{c, 1}(0)$.
The sets~$Y_c$ and~$\tY_c$ are distinguished by the fact that~$Y_c$ contains in its boundary the common landing point of the external rays~$R_c(7/24)$ and~$R_c(17/24)$, denoted~$\omega(c)$, and that~$\tY_c$ contains in its boundary the common landing point of the external rays~$R_c(5/24)$ and~$R_c(19/24)$.
The map~$f_c^3$ maps each of the sets~$Y_c$ and~$\tY_c$ biholomorphically to~$P_{c, 1}(0)$.
Thus, if we put
\[
\begin{array}{cccl}
  g_c & : Y_c \cup \tY_c & \to & P_{c,1}(0)\\
      &    z & \mapsto & g_c(z) \= f_c^{3}(z),
 \end{array}
\]
then
$$ \Lambda_c = \bigcap_{n\in \N} g_c^{-n}(\cl{P_{c,1}(0)}). $$
The rest of this section is dedicated to prove the following lemma.
\begin{lemm}\label{l:contractions}
There are constants $C_0 > 0$ and $\upsilon_0 > 0$ such that for every 
parameter~$c$ in~$\cP_5(-2)$, every~$\ell$ in~$\N$, and every connected
component~$W$ of $g_c^{-\ell}(P_{c,1}(0))$, we have
$$
\diam(W) \le C_0\exp(-\upsilon_0\ell);
$$
furthermore, for all~$z$ and~$w$ in~$W$ we have
\[
 \left| \frac{Dg_c(z)}{Dg_c(w)} - 1 \right| \le C_0 \exp(-\upsilon_0\ell)
\text{ and }
\log \left| \frac{Dg_c(z)}{Dg_c(w)} \right| \le C_0 \exp(-\upsilon_0\ell).
\]
\end{lemm}
To prove this lemma, we recall some facts from~\cite[\S$4.1$]{CorRiva}.
For a parameter~$c$ in~$\cP_2(-2)$, the open disk~$\hU_c$ containing~$-\beta(c)$ that is bounded by the
equipotential~$2$ and by
\begin{equation*}
\label{eq:pleasant cut}
R_c(7/24) \cup \{ \omega(c) \} \cup R_c(17/24),  
\end{equation*}
contains the closure of~$P_{c,0}(- \beta(c))$ and is disjoint from~$P_{c, 1}(\beta(c))$;
the set $\hW_c \= f_c^{-1}(\hU_c)$ contains the closure of~$P_{c, 1}(0)$
and depends continuously with~$c$ on $\cP_{3}(-2)$.
\begin{lemm}
\label{l:contractions extension}
For every parameter~$c$ in~$\cP_4(-2)$, each of the maps
$$ \psi_c \= (g_c|_{Y_c})^{-1}
\text{ and }
\tpsi_c \= (g_c|_{\tY_c})^{-1} $$
extends biholomorphically to~$\hW_c$.
Moreover, the closures of~$\psi_c(\hW_c)$ and~$\tpsi_c(\hW_c)$ are both included in~$P_{c,1}(0)$.
\end{lemm}
\proof
Fix a parameter~$c$ in~$\cP_4(-2)$.

To prove the first assertion, it is sufficient to show that for~$j$ in~$\{0, 1, 2 \}$ the critical value $c$ is not in $f_c^{-j}(\hW_c)$.
By part~$2$ of Lemma~\ref{l:auxiliary para-puzzle pieces}, the critical value~$c$ is in~$P_{c,4}(-\beta(c))$. 
Then for~$i$ in~$\{1, 2, 3 \}$ the point $f_c^i(c)$ is in the set $P_{c,1}(\beta(c))$ that is disjoint from~$\hU_c$.
Using $\hW_c = f_c^{-1}(\hU_c)$, we conclude the proof of the extension.

To prove the second assertion, we use the fact that~$f_c(Y_c) = f_c(\tY_c)$ and
that~$f_c^2(Y_c)$ is contained in~$P_{c, 1}(\beta(c))$ (\emph{cf.}, 
proof of~\cite[Lemma~$3.5$]{CorRiva}).
Denote by~$\tU_c$ the open disk  containing 0 that is bounded by the
equipotential 2, the point~$\talpha(c)$ and the external rays landing at
$\talpha(c)$.
Observe that $\hU_c\subset \tU_c$ and thus, that~$\hW_c$ is contained in the connected set~$f_c^{-1}(\tU_c)$.
The set~$f_c^{-1}(\tU_c)$ is contained in the set containing~$\beta(c)$ and that is bounded by the equipotential~$1$, by the preimage~$\alpha_1(c)$ of~$\talpha(c)$ contained in~$P_{c, 1}(- \beta(c))$, and by the external rays~$R_c(5/12)$ and~$R_c(7/12)$ that land at~$\alpha_1(c)$.
In particular, $f_c^{-1}(\tU_c)$ is disjoint from~$P_{c, 4}(- \beta(c))$.
This implies that~$f_c^{-2}(\tU_c)$ has~$2$ connected components, one that is disjoint from~$P_{c, 1}(\beta(c))$ and the other one that contains~$f_c^2(Y_c)$; the closure of the latter is contained in~$P_{c, 0}(\beta(c))$.
Since~$f_c^2(P_{c, 1}(0))$ contains~$P_{c, 0}(\beta(c))$, we conclude that the closures of the connected components of~$f_c^{-4}(\tU_c)$ containing~$Y_c$ and~$\tY_c$ are both contained in~$P_{c, 1}(0)$.
This proves that the closures of~$\psi_c(\hW_c)$ and~$\tpsi_c(\hW_c)$ are both contained in~$P_{c, 1}(0)$.
\endproof

\begin{proof}[Proof of Lemma~\ref{l:contractions}]
By part~$1$ of Lemma~\ref{l:auxiliary para-puzzle pieces}, the closure of~$\cP_5(-2)$ is a compact set included in $\cP_4(-2)$.
Since~$P_{c,1}(0)$ and~$\hW_{c}$ vary continuously 
with~$c$ in $\cP_4(-2)$ (\emph{cf.}, \cite[Lemma~$2.5$]{CorRiva}), the same
holds for
$$ W_c \= \psi_c(\hW_c)
\text{ and }
\tW_c \= \tpsi_c(\hW_c). $$
Therefore, by Lemma~\ref{l:contractions extension} we have
$$ A \= \inf_{c\in \cP_5(-2)} \min \left\{ \text{mod}(\hW_c\setminus
\cl{W_c}), \text{mod}(\hW_c\setminus \cl{\tW_c}) \right\} > 0,$$
$$ \Xi_0 \= \inf_{c\in \cP_5(-2)} \dist(\partial \hW_{c}, P_{c,1}(0)) > 0, $$
$$ \Xi_1 \= \sup_{c\in \cP_5(-2)} \diam(P_{c,1}(0)) < + \infty, $$
and
$$ \Xi_2 \= \sup_{c \in \cP_5(-2)} \sup_{z \in \C, |z| \le 2 \Xi_1} |Df_c^3(z)| < + \infty. $$

For an open topological disk~$U$ in~$\C$, denote by~$\dist_{U}$ the Poincar{\'e} distance on~$U$.
Note that there is a constant~$\hC > 0$ that only depends on~$\Xi_0$, such that for every~$c$ in $\cP_5(-2)$ the Euclidean and Poincar{\'e} distances on~$\hW_c$ are comparable by a factor of~$\hC$ on~$P_{c,1}(0)$, see for example~\cite[Lemma~A.$8$]{Mil06}.
On the other hand, by Pick's Theorem (see for instance~\cite{Mil06}), for every parameter~$c$ in~$\cP_4(-2)$ the maps~$\psi_c$ and~$\tpsi_c$ are isometries for the Poincar{\'e} distances on~$\hW_c$ and on~$W_c$ and~$\tW_c$, respectively. 
Again by Pick's Theorem, each of the inclusion maps from~$W_c$ and~$\tW_c$ into~$\hW_c$ are contractions for the corresponding Poincar{\'e} distances.
It follows that there is $\upsilon_0>0$ that only depends on~$A$, such that each of these inclusions contracts by a factor at least~$\exp(-\upsilon_0)$.
Thus, for every parameter~$c$ in~$\cP_5(-2)$ and all~$x$ and~$y$ in~$\hW_c$, we have
$$ \dist_{\hW_c}(\psi_c(x),\psi_c(y))
\le
\exp(-\upsilon_0) \dist_{\hW_c}(x,y) $$
and
$$ \dist_{\hW_c}(\tpsi_c(x),\tpsi_c(y))
\le
\exp(-\upsilon_0) \dist_{\hW_c}(x,y). $$

Let~$\ell \ge 1$ be an integer and~$W$ a connected component of~$g_c^{-\ell}(P_{c, 1}(0))$.
Note that~$\left( g_c^{\ell}|_{W} \right)^{-1}$ extends to a holomorphic map~$\psi$ defined on~$\hW_c$ that can be written as the composition of~$\ell$ maps in~$\{ \psi_c, \tpsi_c \}$.
Thus,
\begin{equation*}
\diam(W)
=
\diam(\psi(P_{c, 1}(0)))
\le
\hC^2 \exp( - \upsilon_0 \ell) \diam(P_{c, 1}(0)).
\end{equation*}
This proves the first desired estimate with~$C_0 = \hC^2 \Xi_1$.

To prove the remaining estimates, note that for each point~$w$ in~$Y_c \cup \tY_c$ and every~$z$ in~$\C$ satisfying~$|z| = 2 \Xi_1$, we have
$$ |z - w| \ge \Xi_1
\text{ and }
|Df_c^3(z) - Df_c^3(w)| \le 2 \Xi_2. $$
So for each~$w$ in~$Y_c \cup \tY_c$ the maximum principle applied to the holomorphic function
$$ z \mapsto \frac{Df_c^3(z) - Df_c^3(w)}{z - w} $$
and to~$\{ z \in \C \mid |z| \le 2 \Xi_1 \}$, gives for every~$z$ in~$Y_c \cup \tY_c$
$$ |Dg_c(z) - Dg_c(w)|
=
|Df_c^3(z) - Df_c^3(w)|
\le
2 \Xi_2 \Xi_1^{-1} |z - w|. $$
On the other hand, since each of the maps~$\psi_c$ and~$\tpsi_c$ is a contraction for the Poincar{\'e} distance on~$\hW_c$, by the definition of~$\hC$ we have for every~$w$ in~$Y_c \cup \tY_c$ that~$|Dg_c(w)|^{-1} \le \hC^2$.
We conclude that for all~$z$ and~$w$ in~$Y_c$ or in~$\tY_c$, we have
$$ \left| \frac{Dg_c(z)}{Dg_c(w)} - 1 \right|
\le
2 \hC^2 \Xi_2 \Xi_1^{-1} |z - w|. $$
Together with the first estimate of the lemma, this implies the second and third estimates with~$C_0 = (2 \hC^2 \Xi_2 \Xi_1^{-1}) (\hC^2 \Xi_1)$.
\end{proof}

\subsection{Parameters}
\label{ss:Parameters}
The parameter we use to prove the Main Theorem is chosen from a set introduced in~\cite[Proposition~$3.1$]{CorRiva}.
In this subsection we recall the definition of this parameter set, and give some dynamical properties of the corresponding maps.

Given an integer~$n \ge 3$, let~$\cK_n$ be the set of all those real
parameters~$c$ such that the following properties hold:
\begin{enumerate}[1.]
\item
We have $c < 0$, and for each~$j$ in~$\{1, \ldots, n - 1 \}$ we have~$f_c^j(c) >
0$.
\item 
For every integer~$k \ge 0$, we have
$$ f_c^{n + 3k + 1}(c) < 0
\text{ and }
f_c^{n + 3k + 2}(c) > 0. $$
\end{enumerate}

Note that for a parameter~$c$ in~$\cK_n$ the critical point of~$f_c$ cannot be
asymptotic to a periodic point,  see~\cite[\S$8$]{MilThu88}.
This implies that all the periodic points of~$f_c$ in~$\C$ are hyperbolic
repelling and therefore that~$K_c = J_c$, see~\cite{Mil06}.
On the other hand, we have~$f_c(c) > c$ and the interval~$I_c = [c, f_c(c)]$ is
invariant by~$f_c$.
This implies that~$I_c$ is contained in~$J_c$ and hence that for every real
number~$t$ we have~$P_c^{\R}(t) \le P_c^{\C}(t)$.
Note also that~$f_c|_{I_c}$ is not renormalizable, so~$f_c$ is topologically
exact on~$I_c$, see for example~\cite[Thoerem~III.$4$.$1$]{dMevSt93}.

Since for~$c$ in~$\cK_n$ the critical point of~$f_c$ is not periodic, for every integer~$k \ge 0$ we have~$f_c^{n + 3k}(c) \neq 0$.
Thus, we can define the sequence~$\iota(c)$ in~$\{0, 1 \}^{\N_0}$ for each~$k \ge 0$ by
$$ \iota(c)_k
\=
\begin{cases}
0 & \text{ if }  f_c^{n + 3k}(c) < 0; \\
1  & \text{ if } f_c^{n + 3k}(c) > 0.
\end{cases} $$

\begin{prop}
\label{p:ps}
For each integer~$n \ge 3$, the set~$\cK_n$ is a compact subset of
$$ \cP_n(-2) \cap (-2, -3/4), $$
and for every sequence~$\underline{x}$ in~$\{0,1\}^{\N_0}$ there is a unique 
parameter~$c$ in~$\cK_n$ such that~$\iota(c) = \underline{x}$.
Finally, for each~$\delta > 0$ there is~$n_0 \ge 3$ such that for each
integer~$n \ge n_0$ the set~$\cK_n$ is contained in the interval~$(-2, -2 +
\delta)$.
\end{prop}

Recall that for an open subset~$G$ of~$\C$ and a univalent map~$f : G \to \C$, the \emph{distortion of~$f$} on a subset~$C$ of~$G$ is by definition
$$ \sup_{x, y \in C} |Df(x)|/|Df(y)|. $$
The following is a uniform distortion bound for parameters as in the previous proposition.
\begin{lemm}[\cite{CorRiva}, Lemma~$4.3$]
\label{l:distortion to central 0}
There is a constant~$\Delta_0 > 1$ such that for each integer~$n \ge 4$ and each
parameter~$c$ in~$\cK_n$ the following properties hold: For each integer~$m \ge 1$ and each connected component~$W$ of~$f_c^{-m}(P_{c, 1}(0))$ on which~$f_c^m$ is univalent, $f_c^m$ maps a neighborhood
of~$W$ biholomorphically to~$\hW_c$ and the distortion of this map on~$W$ is bounded
by~$\Delta_0$.
\end{lemm}

\subsection{Induced map and pressure function}
\label{ss:induced map}
Let~$n \ge 5$ be an integer and~$c$ a parameter in~$\cK_n$.
Throughout the rest of this subsection we put $\hV_c \= P_{c, 4}(0)$.
Note that the critical value~$c$ of~$f_c$ is in~$P_{c, n}(- \beta(c))$ (part~$2$ of Lemma~\ref{l:auxiliary para-puzzle pieces} and Proposition~\ref{p:ps}), so the closure of
$$ V_c \= P_{c, n + 1}(0) = f_c^{-1}(P_{c, n}(- \beta(c))) $$
is contained in~$\hV_c = f_c^{-1}(P_{c, 3}(- \beta(c)))$, \emph{cf}. \cite[part~$1$ of Lemma~$3.2$]{CorRiva}.

Let~$D_c$ be the set of all those points~$z$ in~$V_c$ for which there is an integer~$m \ge 1$ such that~$f_c^m(z)$ is in~$V_c$.
For~$z$ in~$D_c$ we denote by~$m_c(z)$ the least integer~$m$ with this property, and call it the \emph{first return time of~$z$ to~$V_c$}.
The \emph{first return map to~$V_c$} is defined by
$$ \begin{array}{rcl}
F_c : D_c & \to & V_c \\
z & \mapsto & F_c(z) \= f_c^{m_c(z)}(z).
\end{array} $$
It is easy to see that~$D_c$ is a disjoint union of puzzle pieces; 
so each connected component of~$D_c$ is a puzzle piece.
Note furthermore that in each of these puzzle pieces~$W$, 
the return time function~$m_c$ is constant; denote the common value
of~$m_c$ on~$W$ by~$m_c(W)$.

Denote by~$\fD_c$ the collection of connected components of~$D_c$ and
by~$\fD_c^{\R}$ the sub-collection of~$\fD_c$ of those sets intersecting~$\R$.
For each~$W$ in~$\fD_c$ denote by~$\phi_W : \hV_c \to V_c$ 
the extension of~$F_c|_{W}^{-1}$ given by~\cite[Lemma~$6.1$]{CorRiva}. Given an
integer~$\ell \ge 1$ we denote by~$E_{c, \ell}$ (resp. $E_{c, \ell}^{\R}$) the
set of all words of length~$\ell$ in the alphabet~$\fD_c$ (resp. $\fD_c^{\R}$).
Again by~\cite[Lemma~$6.1$]{CorRiva}, for each integer~$\ell \ge 1$ and
each word~$W_1 \cdots W_\ell$  in~$E_{c, \ell}$ the composition
$$ \phi_{W_1 \cdots W_\ell} = \phi_{W_1} \circ \cdots \circ \phi_{W_\ell} $$
is defined on~$\hV_c$.
We also put
$$ m_c(W_1 \cdots W_\ell) = m_c(W_1) + \cdots + m_c(W_\ell). $$

For~$t, p$ in~$\R$ and an integer~$\ell \ge 1$ put
$$ Z_{\ell}(t, p)
\=
\sum_{\underline{W} \in E_{c, \ell}} \exp(-m_c(\underline{W}) p) \left( \sup \{
|D\phi_{\underline{W}}(z) | \mid z \in V_c \} \right)^t $$
and
$$ Z_{\ell}^{\R}(t, p)
\=
\sum_{\underline{W} \in E_{c, \ell}^{\R}} \exp(-m_c(\underline{W}) p) \left(
\sup \{ |D\phi_{\underline{W}}(z) | \mid z \in V_c \} \right)^t. $$
For a fixed~$t$ and~$p$ in~$\R$ 
the sequence
$$ \left(\frac{1}{\ell} \log Z_{\ell}(t, p) \right)_{\ell = 1}^{+ \infty}
\left( \text{resp. } \left(\frac{1}{\ell}  \log Z_{\ell}^{\R}(t, p) \right)_{\ell = 1}^{+ \infty} \right) $$
converges to the pressure function of~$F_c$ (resp. $F_c|_{D_c \cap \R}$) for the
potential~$- t \log |DF_c| - p m_c$; 
we denote it by~$\sP_c^{\C}(t, p)$ (resp. $\sP_c^{\R}(t, p)$).
On the set where it is finite, the function~$\sP_c^{\C}$ (resp.~$\sP_c^{\R}$) so
defined is strictly decreasing in each of its variables.

\section{The~$2$ variables series}
\label{s:2 variables series}
We start this section stating a stronger version of the Main Theorem in~\S\ref{ss:main technical theorem}.
The rest of this section is dedicated to estimate, for a real parameter~$c$ in~$\bigcup_{n = 6}^{+ \infty} \cK_n$ satisfying some mild hypotheses, a certain ``postcritical series'' in terms of an abstract~$2$ variables series (Proposition~\ref{p:2 variables series} in~\S\ref{ss:2 variables series}).
The postcritical series is used in~\S\ref{s:finale} to estimate the geometric pressure function.
The definition of the~$2$ variables series is based on an approximation of the derivatives~$(Df_c^n(c))_{n = 1}^{+ \infty}$, using the derivatives of~$g_c$ at its fixed points~$p(c)$ and~$\wtp(c)$.
This approximation, which is more precise than a direct application of the Koebe principle, incorporates an estimate of the corresponding distortion 
constants (Proposition~\ref{p:Improved distortion estimate} in~\S\ref{ss:Improved distortion estimate}).
This estimate is given in terms of the total distortion of the~$2$ homoclinic orbits of~$g_c$ connecting~$p(c)$ and~$\wtp(c)$.

\subsection{Main Technical Theorem}
\label{ss:main technical theorem}
In this subsection we state the Main Technical Theorem from which the Main Theorem follows directly.
The rest of the paper is dedicated to the proof of the Main Technical Theorem. 

Let~$c$ be a parameter in~$\bigcup_{n = 6}^{+ \infty} \cK_n$.
An invariant probability measure supported on~$I_c$ (resp.~$J_c$) is said to be
an \emph{equilibrium state of~$f_c|_{I_c}$} (resp.~$f_c$) \emph{for the
potential~$-\log |Df_c|$}, 
if the supremum defining~$P_c^{\R}(t)$ (resp.~$P_c^{\C}(t)$) is attained at this
measure.
Given~$t > 0$ and a real number~$p$ we say a measure~$\mu$ is \emph{$(t, p)$\nobreakdash-conformal for~$f_c|_{I_c}$} (resp.~$f_c$), if for every subset~$U$ of~$I_c$ (resp.~$J_c$) on which~$f_c|_{I_c}$ (resp.~$f_c$) is injective we have
\begin{multline*}
\mu(f_c|_{I_c}(U)) = \exp(p) \int_U |Df_c|^{t} d\mu
\\
\left( \text{ resp. } \mu(f_c(U)) = \exp(p) \int_U |Df_c|^{t} d\mu \right).
\end{multline*}
In the case where~$P_c^{\R}(t) = 0$ (resp. $P_c^{\C}(t) = 0$), a $(t, 0)$\nobreakdash-conformal measure is simply called \emph{conformal}.

For each~$c$ in~$\cP_3(-2)$ denote by~$p(c)$ the unique fixed point of~$g_c$ in~$Y_c$ and by~$\wtp(c)$ the unique fixed point of~$g_c$ in~$\tY_c$.
Each of the functions
$$ p : \cP_3(-2) \to \C
\text{ and }
\wtp : \cP_3(-2) \to \C $$
so defined is holomorphic.
\begin{generic}[Main Technical Theorem]
\label{t:non-differentiable transition}
There is $n_1 \ge 6$ such that for every integer $n\ge n_1$ there
are a parameter~$c$ in $\cK_n$, an integer $q \ge 3$, and real numbers~$\kappa$ in~$[1, 2]$ and $\Delta \ge 1$, such that the following properties are satisfied.
Put 
$$
t_*
\=
\frac{2 \log 2}{\log \frac{| Dg_c(p(c)) |}{ |Dg_c(\wtp(c))|}}
\text{ and }
t_0
\=
\frac{q-2}{q-1} \cdot t_*,
$$
and define the functions~$\delta^+$, $\delta^-$, $p^+$,
$p^- : (t_0, + \infty) \to \R$, by
\[
 \delta^+(t) \= 
\begin{cases}
\frac{2 \log 2}{3} \cdot 2^{- q \left( \frac{\kappa t_*}{q\left( t_* - t \right)} - 1 \right)^2}
& \text{if } t \in (t_0,t_*);
\\
0 & \text{if } t \ge t_*;
\end{cases}
\]
\[
 \delta^-(t)
\=
\begin{cases}
\frac{\log 2}{3} \cdot 2^{- q \left( \frac{\kappa t_*}{q\left( t_* - t \right)} + \Delta \right)^2}
& \text{if } t \in (t_0,t_*);
\\
0 & \text{if } t\ge t_*;
\end{cases}
\]
\[
p^+(t)\= - t \frac{\chicrit}{2} + \delta^+(t),
\text{ and }
p^-(t)\= - t \frac{\chicrit}{2} + \delta^-(t).
\]
Then, $\chicrit > 0$, for~$t > t_0$ we have
$$
p^-(t)  \le P_c^\R(t) \le  P_c^{\C}(t) \le p^+(t),
$$
and for $t \ge t_*$ there is no equilibrium state of~$f_c|_{I_c}$
(resp.~$f_c|_{J_c}$) for the potential~$- t \log |Df_c|$ and we have
$$
\sP_c^\R\left(t, - t \frac{\chicrit}{2}\right)
\le 
\sP_c^{\C}\left(t, - t\frac{\chicrit}{2}\right)
<
0. $$
Moreover, for~$t \ge t_*$ and for~$p$ in~$\R$ the following properties hold:
\begin{enumerate}
\item[1.]
If~$p \ge - t \chicrit/2$, then there is a unique $(t, p)$\nobreakdash-conformal probability measure for~$f_c|_{I_c}$ (resp.~$f_c$) supported on~$I_c$ (resp.~$J_c$).
Moreover, this measure is dissipative, purely atomic, and supported on the backward orbit of~$z = 0$.
\item[2.]
If~$p < - t \chicrit/2$, then there is no $(t, p)$\nobreakdash-conformal probability measure for~$f_c|_{I_c}$ (resp.~$f_c$) supported on~$I_c$ (resp.~$J_c$).
\end{enumerate}
\end{generic}

\subsection{Improved distortion estimate}
\label{ss:Improved distortion estimate}
The purpose of this subsection is to prove Proposition~\ref{p:Improved distortion estimate}, below.
For the statement, define for each~$c$ in~$\cP_3(-2)$ the itinerary map
$$ \iota_c : \Lambda_c \to \{ 0, 1 \}^{\N_0}, $$
for~$x$ in~$\Lambda_c$ and~$k$ in~$\N_0$, by
$$ \iota_c(x)_k
\=
\begin{cases}
0 & \text{ if }  g_c^k(x) \in Y_c; \\
1 & \text{ if } g_c^k(x) \in \tY_c.
\end{cases} $$
We recall from \cite[\S$3.3$]{CorRiva} that the map~$\iota_c$ conjugates
the action of~$g_c$ on~$\Lambda_c$ to the action of the shift map on $\{ 0, 1 \}^{\N_0}$.
Moreover, if $c$ is real, then~$\Lambda_c$ is contained in~$\R$, the sets
$$ Y_c \cap \R
\text{ and }
f_c(Y_c \cap \R) = f_c(\tY_c \cap \R) $$
are both contained in the negative real numbers, and the sets
$$ \tY_c \cap \R
\text{ and }
f_c^2(Y_c \cap \R) = f_c^2(\tY_c \cap \R) $$
are both contained in the positive real numbers.
It follows that for~$c$ in~$\cK_n$ the point~$f_c^n(c)$ is in~$\Lambda_c$ and the sequence $\iota(c)$ defined in~\S\ref{ss:Parameters} is equal to $\iota_c(f_c^n(c))$.
Finally, for each~$\underline{x}$ in~$\{0, 1\}^{\N_0}$ define the function
\[
\begin{array}{cccl}
I_{\underline{x}} :  & \cP_3(-2) & \to & \C \\
      & c & \mapsto & I_{\underline{x}}(c) \= \iota_c^{-1}(\underline{x}).
\end{array} 
\]
By a normality argument the function  $I_{\underline{x}}$ is holomorphic.
\begin{prop}[Improved distortion estimate]
\label{p:Improved distortion estimate}
There are analytic functions
$$ \zeta : \cP_5(-2) \to (0, + \infty)
\text{ and }
\tzeta : \cP_5(-2) \to (0, + \infty), $$
and constants~$C_1 > 0$ and~$\upsilon_1 > 0$, such that for every integer~$n \ge 5$ and every parameter~$c$ in~$\cK_n$ the following property holds: Let~$m$ and~$m'$ be positive integers and let
$$ \underline{x} = (x_j)_{j = 0}^{+ \infty}
\left( \text{resp. $\underline{\widetilde{x}} = ( \widetilde{x}_j)_{j = 0}^{+ \infty}$} \right) $$
be a sequence in~$\{0, 1 \}^{\N_0}$ such that for~$j$ in~$\{0, \ldots, m - 1 \}$ we have~$x_j = 0$ (resp. $\widetilde{x}_j = 1$) and such that for~$j$ in~$\{0, \ldots, m' - 1 \}$ we have~$x_{m + j} = 1$ (resp. $\widetilde{x}_{m + j} =0$).
Then
$$ \left| \log \frac{|Dg_c^m(I_{\underline{x}}(c))|}{|Dg_c(p(c))|^m} - \log
\zeta(c) \right|
\le
C_1 \exp( - \min \{m, m'\} \upsilon_1 ) $$
and
$$ \left| \log
\frac{|Dg_c^m(I_{\underline{\widetilde{x}}}(c))|}{|Dg_c(\wtp(c))|^m} - \log
\tzeta(c) \right|
\le
C_1  \exp( - \min \{m, m'\} \upsilon_1 ). $$
\end{prop}

The proof of this proposition is at the end of this subsection.

For each integer~$\ell$ in~$\N_0$, let
$$ \underline{x}^{\ell} = (x^\ell_j)_{j = 0}^{+ \infty}
\text{ and }
\underline{\widetilde{x}}^{\ell} = (\widetilde{x}^{\ell}_j)_{j = 0}^{+ \infty} $$
be the sequences in~$\{0, 1 \}^{\N_0}$ defined for each~$j$ in~$\N_0$ by
$$ x^{\ell}_j =
\begin{cases}
0 & \text{if } j \le \ell-1; \\
1 & \text{if } j \ge \ell;
\end{cases} $$
and
$$ \widetilde{x}^{\ell}_j =
\begin{cases}
1 & \text{if } j \le \ell-1; \\
0 & \text{if } j \ge \ell.
\end{cases} $$
Observe that for every~$c$ in~$\cP_3(-2)$ and for every~$\ell$ in~$\N$, the points~$I_{\underline{x}^{\ell}}(c)$ and $p(c)$ are in the same connected component of~$g_c^{- \ell}(P_{c,1}(0))$, and the same holds for $I_{\underline{\widetilde{x}}^{\ell}}(c)$ and~$\wtp(c)$.
Thus, the following is a direct consequence of Lemma~\ref{l:contractions}.

\begin{coro}\label{c:exponential bound}
Let $C_0>0$ and $\upsilon_0 > 0$ be the constants given by Lemma~\ref{l:contractions}. 
Then, for every parameter $c$ in~$\cP_5(-2)$ and for every $\ell$ in $\N$  we have
$$
 \left| \frac{Dg_c(I_{\underline{x}^{\ell}}(c))}{Dg_c(p(c))} - 1 \right| \le C_0 \exp(-\upsilon_0\ell),
\left| \frac{Dg_c(I_{\underline{\widetilde{x}}^{\ell}}(c))}{Dg_c(\wtp(c))} - 1 \right| \le C_0 \exp(-\upsilon_0\ell),
$$
$$
\log  \left| \frac{Dg_c(I_{\underline{x}^{\ell}}(c))}{Dg_c(p(c))} \right|
\le
C_0 \exp(-\upsilon_0\ell)
\text{ and }
\log \left| \frac{Dg_c(I_{\underline{\widetilde{x}}^{\ell}}(c))}{Dg_c(\wtp(c))} \right|
\le
C_0 \exp(-\upsilon_0\ell).
$$
\end{coro}

\begin{lemm}[Homoclinic distortion]
\label{l:homoclinic distortion}
For every parameter~$c$ in~$\cP_5(-2)$ the limits
$$ \zeta(c)
\=
\lim_{m \to + \infty} \prod_{\ell = 1}^{m}
\frac{|Dg_c(I_{\underline{x}^{\ell}}(c))|}{|Dg_c(p(c))|}
\text{ and }
\tzeta(c)
\=
\lim_{m \to + \infty} \prod_{\ell = 1}^{m}
\frac{|Dg_c(I_{\underline{\widetilde{x}}^{\ell}}(c))|}{|Dg_c(\wtp(c))|},$$
exist and depend analytically with~$c$ on~$\cP_5(-2)$.
\end{lemm}
\proof
We prove the existence of the first limit and its analytic dependence on~$c$; the proof of the analogous assertions for the second limit are similar.

Denote by $\log$ the logarithm defined in the open disk of~$\C$ of radius~$1$ centered at~$z = 1$.
By Corollary~\ref{c:exponential bound}, there is $\ell_0$ in $\N$ such that for every $\ell\ge \ell_0$ and every~$c$ in~$\cP_5(-2)$ we have
$\left| \frac{Dg_c(I_{\underline{x}^{\ell}})(c)}{Dg_c(p(c))} - 1 \right| < 1$, so the logarithm~$\log \frac{Dg_c(I_{\underline{x}^{\ell}})(c)}{Dg_c(p(c))}$ is defined.
Corollary~\ref{c:exponential bound} also implies that the sum
$$ \sum_{\ell = \ell_0}^{+\infty} \log \frac{Dg_c(I_{\underline{x}^{\ell}})(c)}{Dg_c(p(c))} $$
exists and is a holomorphic function of~$c$ on~$\cP_5(-2)$.
Exponentiating, we obtain that the infinite product 
$\prod_{\ell = \ell_0}^{+\infty} \frac{Dg_c(I_{\underline{x}^{\ell}})(c)}{Dg_c(p(c))}$ exists and is holomorphic on~$\cP_5(-2)$.
This implies that the infinite product starting from $\ell=1$ also exists and is holomorphic on~$\cP_5(-2)$.
Taking modulus we conclude the proof.
\endproof

\proof[Proof of Proposition~\ref{p:Improved distortion estimate}]
Let~$C_0$ and~$\upsilon_0$ be the constants given by Lemma~\ref{l:contractions} and let
$$ \zeta : \cP_5(-2) \to (0, + \infty)
\text{ and }
\tzeta : \cP_5(-2) \to (0, + \infty) $$
be the continuous functions given by Lemma~\ref{l:homoclinic distortion}.

We only prove the first inequality, the other inequality being similar.
We have
\begin{multline*}
 \log \frac{|Dg_c^m(I_{\underline{x}}(c))|}{|Dg_c(p(c))|^m} - \log
\zeta(c)
\\ =
\sum_{j = 0}^{m-1} 
\log \frac{|Dg_c(g_c^{j}(I_{\underline{x}}(c)))|}{|Dg_c(p(c))|}
- 
\lim_{\widetilde{m}\to +\infty} \sum_{\ell = 1}^{\widetilde{m}} \log
\frac{|Dg_c(I_{\underline{x}^{\ell}}(c))|}{|Dg_c(p(c))|}
\\ =
\sum_{j = 0}^{m-1} 
\log \frac{|Dg_c(g_c^{j}(I_{\underline{x}}(c)))|}{|Dg_c(p(c))|}
-\sum_{\ell = 1}^{m} 
\log \frac{|Dg_c(I_{\underline{x}^{\ell}}(c))|}{|Dg_c(p(c))|}
\\
- \lim_{\widetilde{m}\to +\infty} \sum_{\ell = m + 1}^{\widetilde{m}} 
 \log \frac{|Dg_c(I_{\underline{x}^{\ell}}(c))|}{|Dg_c(p(c))|}.
\end{multline*}
Notice that for every~$j$ in~$\{ 0, \ldots, m - 1 \}$ we have $g_c^j(I_{\underline{x}^{m}}(c))= I_{\underline{x}^{m-j}}(c)$, and that the points 
$g_c^{j}(I_{\underline{x}}(c))$ and $g_c^j(I_{\underline{x}^{m}}(c))$ 
are in the same connected component of $g_c^{-(m + m' - j)}(P_{c,1}(0))$.
Using Lemma~\ref{l:contractions} repeatedly, we get
\begin{multline*}
\left| \sum_{j = 0}^{m-1} \log \frac{|Dg_c(g_c^{j}(I_{\underline{x}}(c)))|}{|Dg_c(p(c))|}
- \sum_{\ell = 1}^{m} \log \frac{|Dg_c(I_{\underline{x}^{\ell}}(c))|}{|Dg_c(p(c))|} \right|
\\ =
\left| \sum_{j = 0}^{m-1} \log \frac{|Dg_c(g_c^{j}(I_{\underline{x}}(c)))|}{|Dg_c(p(c))|}
-\sum_{j = 0}^{m-1} \log \frac{|Dg_c(I_{\underline{x}^{m-j}}(c))|}{|Dg_c(p(c))|} \right|
\\ =
\left| \sum_{j = 0}^{m-1} \log \frac{|Dg_c(g_c^{j}(I_{\underline{x}}(c)))|}{|Dg_c(I_{\underline{x}^{m-j}}(c))|} \right|
=
\left| \sum_{j = 0}^{m-1} \log \frac{|Dg_c(g_c^{j}(I_{\underline{x}}(c)))|}{|Dg_c(g_c^j(I_{\underline{x}^{m}}(c)))|} \right|
\\ \le
C_0  \sum_{j = 0}^{m-1}\exp(-\upsilon_0(m + m'-j))
\le
\frac{C_0 \exp(-\upsilon_0)}{1-\exp(-\upsilon_0)} \exp(-\upsilon_0 m').
\end{multline*}

On the other hand, by Corollary~\ref{c:exponential bound} we have for every integer $\widetilde{m} \ge m$,
\begin{equation*}
\left| \sum_{\ell = m + 1}^{\widetilde{m}} \log \frac{|Dg_c(I_{\underline{x}^{\ell}}(c))|}{|Dg_c(p(c))|} \right|
\le
 C_0  \sum_{\ell = m + 1}^{\widetilde{m}}\exp(-\upsilon_0 \ell) 
\\ \le
\frac{C_0 \exp(-\upsilon_0)}{1-\exp(-\upsilon_0)} \exp(-\upsilon_0 m).
\end{equation*}
Taking $C_1\= 2 \frac{C_0 \exp(-\upsilon_0)}{1-\exp(-\upsilon_0)}$ and $\upsilon_1 \= \upsilon_0$ we conclude the proof of the proposition.
\endproof

\subsection{The~$2$ variable series}
\label{ss:2 variables series}

For each integer~$n \ge 4$ and for each parameter~$c$ in~$\cK_n$, denote by
$$ \iota(c) \= \iota_c(f_c^n(c)) $$
the itinerary for~$g_c$ in the Cantor set~$\Lambda_c$ of the point~$x = f_c^n(c)$, see~\S\ref{ss:Parameters}.
Furthermore, denote by~$N_c : \N_0 \to \N_0$ the function defined by~$N_c(0) \= 0$ and for~$k$ in~$\N$ by
$$ N_c(k)
\=
\sharp \{ j \in \{0, \ldots, k-1 \} \mid \iota(c)_j = 0 \}, $$
and by~$B_c : \N_0 \to \N_0$ the function defined by~$B_c(0) \= 0$, $B_c(1) \= 1$, and for~$k\ge 2$ by
$$ B_c(k)
\=
1 + \sharp \{ j \in \{0, \ldots, k - 2 \}
 \mid \iota(c)_j \neq \iota(c)_{j + 1} \}. $$
Note that for~$k$ in~$\N$ the function~$B_c(k)$ is equal to the number of blocks of~$0$'s and~$1$'s in the sequence~$(\iota(c)_j)_{j = 0}^{k-1}$.

On the other hand, for each parameter~$c$ in~$\cP_5(-2)$, define
\begin{equation}\label{d:theta n}
\theta(c)
\=
\left| \frac{Dg_c(p(c))}{Dg_c(\wtp(c))} \right|^{1/2},
\xi(c)
\=
- \frac{ \log (\zeta(c)\tzeta(c))}{4\log \theta(c)}
\end{equation}
and the~$2$ variables series~$\Pi_c$ on $[0,+\infty)\times[0,+\infty)$ by
\[
 \Pi_c(\tau,\lambda)
\=
\sum_{k = 0}^{+\infty} 2^{- \lambda k - \tau N_{c}(k) + \tau \xi(c) B_c(k)}.
\]

The purpose of this subsection is to prove the following proposition.
\begin{propalph}
\label{p:2 variables series}
There are constants~$C_2 > 1$ and~$\upsilon_2 > 0$ such that for every integer~$n \ge 6$ the following property holds.
Let~$c$ be a parameter in~$\cK_n$ such that~$N_{c}(k)/k \to 0$ 
as~$k \to + \infty$ and such that, if we denote by~$(m_j)_{j = 0}^{+ \infty}$ the
sequence of lengths of the blocks of~$0$'s and~$1$'s in the sequence~$\iota(c)$, then the sum
\[
\sum_{j=0}^{+\infty} \exp(-\min\{m_j,m_{j+1}\} \upsilon_2)
\]
converges.
Then for all $t>0$ and $\delta \ge 0$, we have
\begin{multline*}
\label{eq:2 variables series}
C_2^{-t} \exp(-n\delta)\left( \frac{\exp(\chicrit)}{|Df_c(\beta(c))|}
\right)^{\frac{t}{2}n }
\Pi_c \left( t \frac{\log \theta(c)}{\log 2}, \frac{3 \delta}{\log 2}
\right)
\\ \le
\sum_{k=0}^{+\infty} \exp \left(-(n+3k) \left(- t \frac{\chicrit}{2} + \delta \right) \right)
|Df_c^{n+3k}(c)|^{-\frac{t}{2}}
\\ \le
C_2^t \exp(-n\delta) \left( \frac{\exp (\chicrit)}{|Df_c(\beta(c))|}
\right)^{\frac{t}{2} n }
\Pi_c \left( t \frac{\log \theta(c)}{\log 2}, \frac{3 \delta}{\log 2}
\right).
\end{multline*}
\end{propalph}
The proof of this proposition is at the end of this subsection.
\begin{lemm}
\label{l:bd}
Let~$\Delta_0$ be the constant given by Lemma~\ref{l:distortion to central 0} and let~$C_1 > 0$ and~$\upsilon_1 > 0$ be the constants given by Proposition~\ref{p:Improved distortion estimate}.
Moreover, let~$n \ge 5$ be an integer, let~$c$ be a parameter in~$\cK_n$, and denote by~$(m_j)_{j = 0}^{+ \infty}$ the sequence of lengths of the blocks of~$0$'s and~$1$'s in the sequence~$\iota(c)$.
Then for every integer integer~$k$ in~$\N$ we have
\begin{multline*}
\Delta_0^{-1} \max \left\{ \frac{\zeta(c)}{\tzeta(c)}, \frac{\tzeta(c)}{\zeta(c)} \right\}^{-1/2} \exp\left(-C_1\sum_{j=0}^{B_c(k)-1}\exp(-\min\{m_j, m_{j+1}\} \upsilon_1 )\right) 
\\ \le
\frac{|Dg_c^k(f_c^n(c))|}
{|Dg_c(\wtp (c))|^k \cdot \theta(c)^{2 N_c(k)} \cdot (\zeta(c)\tzeta(c))^{B_c(k)/2}}
\\ \le
\Delta_0 \max \left\{ \frac{\zeta(c)}{\tzeta(c)}, \frac{\tzeta(c)}{\zeta(c)} \right\}^{1/2} \exp\left(C_1\sum_{j=0}^{B_c(k)-1}\exp(-\min\{m_j, m_{j+1}\} \upsilon_1)\right).
\end{multline*}
\end{lemm}
\proof
If the first~$k$ entries of~$\iota(c)$ are equal, then~$B_c(k) = 1$ and the desired assertion follows from Lemma~\ref{l:distortion to central 0}.
Suppose that not all of the first~$k$ entries of~$\iota(c)$ are equal, and let~$k_0$ be the maximal element of~$\{1, \ldots, k \}$ such that
$$ \iota(c)_{k_0 - 1} \neq \iota(c)_{k_0}. $$
Moreover, denote by~$B$ and~$\tB$ the number of blocks of~$0$'s and~$1$'s, respectively, in the sequence~$(\iota(c)_j)_{j = 0}^{k_0-1}$.
We have~$B_c(k_0) = B + \tB$, and
\begin{equation}\label{e:blocks}
|B_c(k_0) - 2B| = |B_c(k_0)-2\tB| \le 1.
\end{equation}

Consider a block of~$0$'s or~$1$'s in~$\iota(c)$ with initial position~$i$ and length~$m$, and let~$m'$ be the length of the next block.
By Proposition~\ref{p:Improved distortion estimate} we have the following~$2$ cases:
If $\iota(c)_i=0$, then
\begin{equation*}
\left| \log \frac{|Dg_c^m(g_c^{i}(f_c^n(c)))|}{|Dg_c(p(c))|^m} - \log
\zeta(c) \right|
\le
C_1 \exp( - \min \{m, m'\} \upsilon_1 );
\end{equation*}
and 
if $\iota(c)_i=1$, then
\begin{equation*}
\left| \log
\frac{|Dg_c^m(g_c^{i}(f_c^n(c)))|}{|Dg_c(\wtp(c))|^m} - \log
\tzeta(c) \right|
\le
C_1  \exp( - \min \{m, m'\} \upsilon_1 ). 
\end{equation*}
Applying these inequalities to each of the blocks of~$0$'s and~$1$'s in~$(\iota(c)_j)_{j = 0}^{k_0-1}$, we obtain
\begin{multline}
\label{e:block time}
 \exp\left(-C_1\sum_{j=0}^{B_c(k_0)-1}\exp(-\min\{m_j, m_{j+1}\}
\upsilon_1)\right)
\\ \le
\frac{|Dg_c^{k_0} (f_c^n(c))|}
{|Dg_c( p(c))|^{N_c(k_0)} |Dg_c(\wtp (c))|^{k_0-N_c(k_0)} \zeta(c)^{B} \tzeta(c)^{\tB}}
\\ \le
\exp\left(C_1\sum_{j=0}^{B_c(k_0)-1}\exp(-\min\{m_j, m_{j+1}\} \upsilon_1)\right).
\end{multline}
Together with~\eqref{e:blocks} this implies the desired chain of inequalities in the case where~$k_0 = k$.
If~$k_0 \le k - 1$, then by Lemma~\ref{l:distortion to central 0} we have
$$ \Delta_0^{-1}
\le
\frac{ |Dg_c^{k - k_0} (g_c^{k_0}(f_c^n(c)))|}{|Dg_c(\wtp(c))|^{k - k_0} \cdot \theta(c)^{2(N_c(k) - N_c(k_0))}}
\le
\Delta_0. $$
This, together with~\eqref{e:blocks}, \eqref{e:block time}, and~$B_c(k) = B_c(k_0) + 1$, implies the desired chain of inequalities.
The proof of the lemma is thus complete.
\endproof

\begin{lemm}
\label{lem:chicrit}
Let~$n \ge 4$ be an integer and let~$c$ be a parameter in~$\cK_n$ such that~$N_{c}(k) / k \to 0$ as~$k \to + \infty$.
Then we have
 \[
\chicrit
=
\frac{1}{3} \log |Dg_c(\wtp(c))|.
 \]
\end{lemm}
\proof
Put~$\whc \= f_c^n(c)$.
For every~$k$ in~$\N$ and every~$j$ in~$\{0,1,2\}$, we have by the chain rule
\begin{eqnarray*}
Df_c^{3k + j} (c)
& = &
Df_c^j ((f_c^{3k})(\whc)) \cdot Df_c^{3k} (\whc) \cdot Df_c^n (c)
\\ & = &
Df_c^j (g_c^{k}(\whc)) \cdot Dg_c^k (\whc) \cdot Df_c^n (c).
\end{eqnarray*}
Since $|Df_c^j ((g_c^{k})(\whc))|$ is bounded independently of~$k$ and~$j$, we have
\begin{equation}
\label{eq:intermediate chicrit}
\chicrit
= 
\liminf_{m\rightarrow +\infty} \frac{1}{m} \log |Df_c^m(c)|
=
\frac{1}{3} \liminf_{k\rightarrow +\infty}\frac{1}{k} \log |Dg_c^k (\whc)|.
\end{equation}
On the other hand, by Lemma~\ref{l:distortion to central 0}, there is a constant $\Delta_0>1$ such that 
for each integer~$k$ in $\N$,
\begin{equation*}
\Delta_0^{-B_c(k)}
\le
\frac{|Dg_c^k(\whc)|} {|Dg_c(\wtp (c))|^{k-N_c(k)} |Dg_c( p(c))|^{N_c(k)}}
\le
\Delta_0^{B_c(k)}.
\end{equation*}
Taking logarithm yields
\begin{multline*}
- B_c(k)\log \Delta_0 + N_c(k) \log \frac{|Dg_c(p(c))|}{|Dg_c(\wtp(c))|}
\\ \le
\log |Dg_c^k(\whc)| - k \log |Dg_c(\wtp(c))|
 \\ \le
B_c(k) \log \Delta_0 + N_c(k) \log \frac{|Dg_c(p(c))|}{|Dg_c(\wtp(c))|}.
\end{multline*}
Since for each~$k$ in~$\N$ we have $B_c(k) \le 2 N_c(k) + 1$, using the hypothesis that 
$N_c(k)/k \to 0$ as $k\to +\infty$, we conclude
$$ \lim_{k\rightarrow +\infty}\frac{1}{k} \log |Dg_c^k (\whc)|
=
\log |Dg_c(\wtp(c))|. $$
Combined with~\eqref{eq:intermediate chicrit}, this completes the proof of the lemma.
\endproof

\begin{lemm}[\cite{CorRiva}, Lemma~$3.6$]
\label{l:landing to central derivative}
There is a constant~$\Delta_1 > 1$, such that for each parameter~$c$
in~$\cP_2(-2)$, each integer~$k \ge 2$, and each point~$y$ in~$P_{c, k}(- \beta(c))$, we have
$$ \Delta_1^{-1} |Df_c(\beta(c))|^k \le |Df_c^k (y)|
\le
\Delta_1 |Df_c(\beta(c))|^k. $$
\end{lemm}

\proof[Proof of Proposition~\ref{p:2 variables series}]
Let~$\Delta_0$ be the constant given by Lemma~\ref{l:distortion to central 0}, let~$C_1$ and~$\upsilon_1$ be the constants given by Proposition~\ref{p:Improved distortion estimate}, and let~$\Delta_1$ be the constant given by Lemma~\ref{l:landing to central derivative}.
Note that by Proposition~\ref{p:Improved distortion estimate} and part~$1$ of Lemma~\ref{l:auxiliary para-puzzle pieces},
\[
\Delta
\=
\sup_{c\in \cP_6(-2)} \max \left\{ \frac{\zeta(c)}{\tzeta(c)}, \frac{\tzeta(c)}{\zeta(c)} \right\}
<
+ \infty.
\]

Let~$n$, $c$, and~$(m_j)_{j = 0}^{+ \infty}$ be as in the statement of the proposition, and put
\[
 \sigma \=  C_1\sum_{j=0}^{+\infty} \exp(-\min\{m_j,m_{j+1}\} \upsilon_1)
\text{ and }
\hC_2\= \Delta_0 \Delta_1 \Delta^{1/2} \exp(\sigma).
\]
Then for every~$k$ in~$\N$ and every~$t > 0$, we have, using
\[
Df_c^{n+3k}(c) = Dg_c^k (f_c^n(c)) \cdot Df_c^n (c)
\]
and combining Lemmas~\ref{l:bd} and~\ref{l:landing to central derivative},
\begin{multline}
\label{e:postcritical derivative}
\hC_2^{-t} \theta(c)^{- 2t N_c(k)} (\zeta(c)\tzeta(c))^{- tB_c(k)/2}
 \\ \le
\frac{|Df_c^{n+3k}(c)|^{-t}}{|Dg_c(\wtp(c))|^{-tk} |Df_c(\beta(c))|^{-tn}}
\\ \le
\hC_2^{t} \theta(c)^{- 2t N_c(k)} (\zeta(c)\tzeta(c))^{- tB_c(k)/2}.
\end{multline}
Since by Lemma~\ref{lem:chicrit} we have
\begin{equation*}
\exp((n+3k)t\chicrit)
=
\exp(nt\chicrit) |Dg_c(\wtp(c))|^{tk},
\end{equation*}
if we multiply each term in the chain of inequalities~\eqref{e:postcritical derivative} by
$$ \left( \frac{\exp(\chicrit)}{|Df_c(\beta(c))|} \right )^{ t n}, $$
then we get
\begin{multline*}
\hC_2^{-t} \left( \frac{\exp(\chicrit)}{|Df_c(\beta(c))|} \right )^{ t n} \theta(c)^{-2t N_c(k)} (\zeta(c)\tzeta(c))^{- tB_c(k)/2} 
\\ \le
\exp((n+3k)t\chicrit) |Df_c^{n+3k}(c)|^{-t}
\\ \le
\hC_2^{t} \left( \frac{\exp(\chicrit)}{|Df_c(\beta(c))|} \right )^{ t n} \theta(c)^{-2t N_c(k)} (\zeta(c)\tzeta(c))^{- tB_c(k)/2}.
\end{multline*}
Taking square roots and then by multiplying by~$\exp(-(n + 3k) \delta)$ in each of the terms of the chain of inequalities above, we obtain
\begin{multline*}
\hC_2^{-t/2} \exp(- n \delta) \left( \frac{\exp(\chicrit)}{|Df_c(\beta(c))|} \right )^{\frac{t}{2} n}
\exp(- 3k \delta)  \theta(c)^{-t N_c(k)} (\zeta(c)\tzeta(c))^{- tB_c(k)/4}
\\ \le
\exp \left(- (n+3k)\left(- t \frac{\chicrit}{2} + \delta\right)\right) |Df_c^{n+3k}(c)|^{-\frac{t}{2}}
\\ \le
\hC_2^{t/2}\exp(- n \delta) \left( \frac{\exp(\chicrit)}{|Df_c(\beta(c))|} \right )
^{\frac{t}{2} n} \exp(- 3k \delta)  \theta(c)^{-t N_c(k)}
(\zeta(c)\tzeta(c))^{- tB_c(k)/4}.
\end{multline*}
Note that when~$k = 0$ this chain of inequalities holds by Lemma~\ref{l:distortion to central 0} and our definition of~$\hC_2$.
Summing over~$k\ge 0$, we obtain the proposition with~$C_2 = \hC_2^{1/2}$.
\endproof

\section{Estimating the~$2$ variables series}
\label{s:reduction}
This section is dedicated to estimate, in an abstract setting, the~$2$ variables series defined in~\S\ref{ss:2 variables series} for a certain itinerary defined in~\S\ref{ss:the itinerary}.
Our main estimate is stated as Proposition~\ref{p:estimating 2 variables series} in~\S\ref{ss:estimating 2 variables series}.
\subsection{The itinerary}
\label{ss:the itinerary}
Given an integer~$\Xi$, let~$q \ge 3$ be a sufficiently large integer such that 
$q + \Xi \ge 1$ and~$2^{q - 1} \ge q + 1 + \Xi$.
Define the quadratic function
$$ \begin{array}{cccl}
 Q : & \R & \to & \R \\
       &  s  & \mapsto & Q(s) \= q s^2
\end{array} $$
and for each real number~$s$ in~$[0,+\infty)$ define the following intervals of~$\R$:
\[
I_s
\=
\left[ 2^{Q(s)}, 2^{Q(s)} +Q(s+1) - Q(s) + \Xi \right)
\]
and
\[
J_s
\=
\left[ 2^{Q(s)} +Q(s+1) - Q(s) + \Xi, 2^{Q(s+1)} \right).
\]
Denote by~$(x_j)_{j = 0}^{+ \infty}$ the sequence in~$\{ 0, 1 \}^{\N_0}$ defined by the property that~$x_j = 0$ if and only if there is an integer~$s \ge 0$ such that~$j + 1$ is in~$I_s$.
Note that the first~$|I_0| = q + \Xi$ entries of~$(x_j)_{j = 0}^{+ \infty}$ are equal to~$0$.
Moreover, define the function $N : \N_0 \to \N_0$, by $N(0) \= 0$, and for~$k$ in~$\N$ by
\[
N(k)
\=
\sharp \left\{ j \in \{0, \ldots, k-1 \} \mid x_j = 0 \right\},
\]
and the function~$B : \N_0 \to \N_0$ by $B(0) \= 0, B(1) \= 1$, and for~$k \ge 2$ by
\[
B(k)
\=
1 + \sharp \left\{ j \in \{ 0, \ldots, k - 2 \} \mid x_j \neq x_{j + 1}
\right\}.
\]
Note that for~$k \ge 1$ the number~$B(k)$ is equal to the number of blocks of~$0$'s and~$1$'s in the
sequence~$(x_j)_{j = 0}^{k-1}$.

Observe that for every~$s$ in~$\N_0$ and every~$k$ in~$J_s$, we have
\begin{equation}
\label{eq:N J}
N(k)
=
\sum_{j = 0}^s |I_j|
=
\sum_{j = 0}^s (Q(j + 1) - Q(j) + \Xi)
 =
Q(s+1) + \Xi \cdot (s+1)
\end{equation}
 and
\begin{equation}
\label{eq:B J}
B(k) = 2(s + 1).
\end{equation}
On the other hand, for each~$s$ in~$\N_0$ and~$k$ in~$I_s$, we have
\begin{equation}
\label{eq:N I}
N(k) = k - (2^{Q(s)} - 1) + Q(s) + \Xi s
\end{equation}
 and
\begin{equation}
\label{eq:B I}
B(k) = 2s + 1.
\end{equation}

\begin{lemm}
\label{l:the itinerary}
The the following properties hold for each real number~$s \ge 0$:
\begin{enumerate}[(a)]
 \item
$2^{Q(s)} + Q(s+1) + \Xi \le 2^{Q(s+1)-1}$.
 \item
$|J_s| \ge 2^{Q(s+1)-1}$.
\end{enumerate}
\end{lemm}
\begin{proof}
Part~$(a)$ with~$s = 0$ is given by our hypothesis~$2^{q - 1} \ge q + 1 + \Xi$.
For~$s > 0$, it follows from this and from the fact that the derivative of the function
$$ s \mapsto 2^{Q(s + 1) - 1} - (2^{Q(s)} + Q(s + 1) + \Xi) $$
is strictly positive on~$[0, + \infty)$.
Part~$(b)$ follows easily from part~$(a)$.
\end{proof}
\subsection{Estimates}
\label{ss:estimating 2 variables series}
Let~$\Xi$ be a given integer and let~$q$, $N$ and~$B$ be as in the previous subsection.
Given a real number~$\xi$ such that~$1 \le \Xi - 2 \xi \le 2$, define the~$2$
variables series~$\Pi$ on $[0,+\infty) \times [0,+\infty)$, by
\[
 \Pi(\tau,\lambda)\= \sum_{k = 0}^{+\infty} 2^{- \lambda k -\tau N(k) + \tau
\xi B(k)}.
\]
This subsection is dedicated to prove the following proposition.
\begin{propalph}
\label{p:estimating 2 variables series}
For every~$\tau \ge 1$ we have,
\[
\Pi(\tau,0) \le 2 \left(2^{\tau\xi} + 1\right).
\]
Furthermore, for each~$\tau$ in~$\left( \tfrac{q-2}{q-1}, 1 \right)$ we have
\[
 \Pi \left( \tau , 2 \cdot 2^{- q \left( \frac{\Xi - 2\xi}{q(1 - \tau)} - 1
\right)^2} \right)
\le
10 \cdot 2^{\tau \xi} + 101,
\]
and for each~$\Delta \ge 1$ we have
\[
2^{\Delta - 4}
\le
\Pi \left(\tau , 2^{- q \left( \frac{\Xi - 2\xi}{q(1 - \tau)} + \Delta
\right)^2} \right).
\]
\end{propalph}
The proof of this proposition is at the end of this subsection.

For every real number~$s$ in~$[0,+\infty)$, define
\[
\lambda(s) \= \frac{1}{|J_s|}.
\]
By part~(b) of Lemma~\ref{l:the itinerary} and the hypothesis~$q \ge 3$, we have $0 < \lambda(s) \le 1/4$.
\begin{lemm}\label{lem:11}
The following properties hold:
\begin{enumerate}[1.]
\item
For $\tau \ge 1$ we have
\[
\Pi(\tau,0) \le 2 \left( 2^{\tau\xi} + 1 \right).
\]
\item
For every real number~$s$ in~$[0,+\infty)$ and every~$\tau$ in~$(1/2,1)$
satisfying~$\tau > \frac{Q(s+1)-1}{Q(s+2)}$, we have
\begin{equation*}
\Pi(\tau,\lambda(s))
\\ \le
1 + 10 \cdot 2^{\tau\xi}
+ 5 \sum_{j=0}^{\lfloor s \rfloor + 1}  2^{(1-\tau)Q(j+1) -\tau (\Xi - 2\xi)
(j+1) }.
\end{equation*}
\item
For every real number~$s$ in~$[0,+\infty)$ and every~$\tau > 0$, we have
\[
\frac{1}{8} 2^{(1-\tau)Q(\lfloor s \rfloor + 1)-\tau(\Xi - 2\xi)(\lfloor s
\rfloor + 1)}
\le
\Pi(\tau,\lambda(s)).
\]
\end{enumerate}
\end{lemm}
\proof
For $\tau > 0$, $\lambda\ge 0$, and $s$ in $\N_0$, define
\begin{equation*}
\label{e:I's}
I_s(\tau, \lambda)
\=
 \sum_{k\in  I_s} 2^{- \lambda k -\tau N(k) +\tau \xi
B(k)}
\end{equation*}
and
\begin{equation*}
\label{e:J's}
J_s(\tau, \lambda)
\=
\sum_{k\in J_s} 2^{- \lambda k -\tau N(k) +\tau \xi B(k)},
\end{equation*}
so that~$\Pi(\tau,\lambda) = 1 + \sum_{s=0}^{+\infty} I_s(\tau, \lambda) +
\sum_{s=0}^{+\infty} J_s(\tau, \lambda)$.

\partn{1}
To prove part~$1$, note that by~\eqref{eq:N I}, \eqref{eq:B I}, and the hypothesis~$\Xi - 2 \xi \ge 1$, for every~$\tau > 0$ and every~$\lambda \ge 0$ we have
\begin{equation}
\label{eq:Is+1}
\begin{split}
\sum_{s=0}^{+\infty} I_s(\tau, \lambda)
& \le
\sum_{s=0}^{+\infty}  \sum_{m = 1}^{|I_s|} 2^{-\tau(Q(s) + \Xi s + m) +
\tau \xi \cdot (2s + 1)}
\\ & =
2^{\tau\xi} \sum_{s=0}^{+\infty} 2^{-\tau(Q(s) + (\Xi - 2 \xi)s) }
\sum_{m = 1}^{|I_{s}|} 2^{-\tau m}
\\ & \le
2^{\tau\xi} \frac{2^{-\tau}}{1-2^{-\tau}} \sum_{s=0}^{+\infty} 2^{-\tau (\Xi - 2\xi)s}
\\ & \le
2^{\tau\xi} \frac{ 2^{-\tau}}{\left(1-2^{-\tau}\right)^2}.
\end{split}
\end{equation}
On the other hand, using~\eqref{eq:N J}, \eqref{eq:B J}, the hypothesis~$\Xi - 2\xi \ge 1$, and that for every $s \ge 0$ we have $|J_s|\le 2^{Q(s + 1)}$, we obtain for every~$\tau \ge 1$
\begin{equation}\label{eq:J}
 \begin{split}
\sum_{s=0}^{+\infty} J_s(\tau, 0)
= & ~
\sum_{s=0}^{+\infty} |J_s|2^{-\tau (Q(s+1) + \Xi(s+1)) + 2 \tau \xi
\cdot (s+1)}
\\ \le & ~
\sum_{s=0}^{+\infty} 2^{-(\tau-1)Q(s+1) - \tau (\Xi - 2 \xi) (s+1) }
\\ \le & ~
\frac{2^{-\tau}}{1-2^{-\tau}}.
\end{split}
\end{equation}
Combining inequalities~\eqref{eq:Is+1} and~\eqref{eq:J}, we get for every~$\tau \ge 1$
\begin{equation*}
\Pi(\tau, 0)
\le
1 + 2^{\tau\xi} \frac{2^{-\tau}}{\left( 1-2^{-\tau} \right)^2}
+ \frac{2^{-\tau}}{1-2^{-\tau}}
\le
2 \left( 2^{\tau\xi} + 1 \right).
\end{equation*}
This is part~$1$ of the lemma.

\partn{2}
Fix~$s$ in $[0,+\infty)$ and set $s_0 \= {\lfloor s \rfloor}$.
We use~\eqref{eq:Is+1} to estimate~$\Pi(\tau, \lambda(s))$.
To estimate~$\sum_{j=0}^{+\infty}J_j(\tau, \lambda(s))$, note that by definition
of~$\lambda(s)$, for each integer~$\ell$ satisfying~$1 \le \ell \le |J_s|$ we
have
\[
\frac{1}{2} \le  2^{- \lambda(s) \ell} \le 1.
\]
On the other hand, the hypothesis $q \ge 3$ implies that the function~$j \mapsto
|J_j|$ is nondecreasing on~$[0, + \infty)$.
Therefore, for each~$j$ in~$\{0, \ldots, s_0 \}$ we have~$|J_j| \le |J_s|$ and
then
\begin{equation}
\label{eq:geometrica1}
\frac{1}{2}|J_j|
\le
\sum_{m = 1}^{|J_j|} 2^{- \lambda(s) m}
\le
|J_j|.
\end{equation}
On the other hand
\begin{equation}\label{eq:geometrica2}
\sum_{m = 1}^{+ \infty} 2^{- \lambda(s) m}
=
\frac{1}{2^{\lambda(s)} - 1}
\le
\frac{1}{\lambda(s) \log 2}
\le
2 |J_s|.
\end{equation}
Note also that, by~\eqref{eq:N J}, \eqref{eq:B J}, and the hypothesis~$q + \Xi \ge 1$, for every~$j$ in~$\N_0$
we have by \eqref{eq:geometrica2} and $|J_s| \le 2^{Q(s+1)}$,
\begin{equation}
\label{eq:10}
\begin{split}
J_j(\tau,\lambda(s))
& =
2^{-\tau(Q(j+1) + \Xi \cdot (j+1)) + 2 \tau \xi \cdot (j+1)} \sum_{k \in J_j}
2^{- \lambda(s)k}\\
&~ \le 
2 |J_s| 2^{-\tau(Q(j+1) + (\Xi - 2 \xi)(j + 1))}
\\ 
&~ \le 
2 \cdot 2^{Q(s + 1)-\tau(Q(j+1) + (\Xi - 2 \xi)(j + 1))} .
\end{split}
\end{equation}
Taking~$j = s_0 + 1$ and using the inequality~$Q(s + 1) \le Q(s_0 + 2)$, we
obtain,
\begin{equation}
\label{eq:intermediate term}
J_{s_0 + 1}(\tau, \lambda(s))
\le
2 \cdot 2^{(1 - \tau) Q(s_0 + 2) - \tau (\Xi - 2 \xi)(s_0 + 2)} .
\end{equation}
On the other hand, our hypothesis $\tau \ge \frac{Q(s + 1) - 1}{Q(s + 2)}$ implies that for $j \ge s_0+2$ we have
\[
Q(s + 1) - \tau Q(j + 1) \le Q(s + 1) - \tau Q(s + 2) \le  1.
\]
So, using the hypothesis~$\Xi - 2 \xi \ge 1$ and summing~\eqref{eq:10} over~$j$
satisfying~$j \ge s_0 + 2$, we obtain
\begin{equation}
\label{eq2}
\sum_{j = s_0 + 2}^{+\infty} J_j(\tau,\lambda(s))
\le
\sum_{j=s_0 + 2}^{+\infty} 2^{2 -\tau (\Xi - 2\xi) (j+1)}
\le
\frac{2^{2-3(\Xi - 2\xi)\tau  }}{1-2^{-\tau }}.
\end{equation}
Now we complete the estimate of~$\sum_{j=0}^{+\infty}J_j(\tau, \lambda(s))$, by
estimating the terms for which~$j$ is in~$\{0, \ldots, s_0 \}$.
From~\eqref{eq:geometrica1}, the first equality in~\eqref{eq:10}, and $|J_j| \le
2^{Q(j+1)}$, we deduce that for every integer~$j$ in~$\{0, \ldots, s_0 \}$ we
have
\begin{equation*}
J_j(\tau, \lambda(s))
 \le
|J_j| \cdot 2^{-\tau (Q(j+1) + (\Xi - 2 \xi)(j + 1))}
 \le
2^{(1-\tau )Q(j+1) -\tau (\Xi - 2 \xi)(j+1)}.
\end{equation*}
Summing over~$j$ in~$\{0, \ldots, s_0 \}$ and using
inequalities~\eqref{eq:intermediate term} and~\eqref{eq2}, we obtain
\begin{multline*}
\sum_{j=0}^{+\infty}J_j(\tau, \lambda(s))
\\ \le
\sum_{j=0}^{s_0} 2^{(1-\tau )Q(j+1) -\tau (\Xi - 2 \xi)(j+1)}
+ 2 \cdot 2^{(1 - \tau) Q(s_0 + 2) - \tau (\Xi - 2 \xi)(s_0 + 2)}
+ \frac{ 2^{2 - 3(\Xi-2\xi)\tau  }}{1-2^{-\tau }}
\\ \le
2 \sum_{j=0}^{s_0 + 1} 2^{(1-\tau )Q(j+1) -\tau (\Xi - 2\xi)(j+1)}
+ \frac{ 2^{2 - 3(\Xi-2\xi)\tau  }}{1-2^{-\tau }}.
\end{multline*}
Together with~\eqref{eq:Is+1} this implies
\begin{multline}
\label{e:preliminary upper bound}
 \Pi(\tau ,\lambda(s))
\le
1 + \frac{ 2^{-\tau }}{\left( 1-2^{-\tau } \right)^2} 2^{\tau \xi}
\\ +
2 \sum_{j=0}^{s_0 + 1} 2^{(1-\tau )Q(j+1) -\tau (\Xi - 2\xi)(j+1)}
+ \frac{ 2^{2 - 3(\Xi - 2\xi)\tau  }}{1- 2^{-\tau } }.
\end{multline}
Using the hypothesis that~$\tau$ is in~$(1/2, 1)$, we have~$\frac{ 2^{-\tau }}{\left( 1-2^{-\tau } \right)^2} \le 10$.
Using in addition the hypotheses~$q \ge 3$ and~$\Xi - 2 \xi \ge 1$, we have
\begin{multline*}
\frac{ 2^{2 - 3(\Xi - 2\xi)\tau  }}{1-2^{-\tau }}
\le
3 \cdot 2^{3 - (\Xi - 2\xi + 3)\tau }
\le
3 \cdot 2^{(1 - \tau)Q(1) - (\Xi - 2\xi)\tau}
\\ \le
3 \sum_{j=0}^{s_0 + 1} 2^{(1-\tau )Q(j+1) -\tau (\Xi - 2\xi)(j+1)}.
\end{multline*}
We obtain part~$2$ of the lemma by combining these estimates
with~\eqref{e:preliminary upper bound}.

\partn{3}
Fix~$s$ in~$[0, + \infty)$ and set~$s_0 \= {\lfloor s \rfloor}$.
By part~(b) of Lemma~\ref{l:the itinerary} and the definition of~$\lambda(s)$, for each~$s$ in~$[0, + \infty)$ we have
\[
\lambda(s) = |J_s|^{-1} \le \frac{1}{2^{Q(s+1)-1}}.
\]
From this inequality and from part~(a) of Lemma~\ref{l:the itinerary}, we obtain  that for every
integer~$j$ in~$\{0, \ldots, s_0 \}$ we have
\begin{equation}
\begin{split}
& \lambda(s) (2^{Q(j)}+Q(j+1)-Q(j) + \Xi - 1)
\\ & \le ~
\lambda(s) (2^{Q(j)} + Q(j+1) + \Xi)
\\ & \le ~
\frac{2^{Q(s)} + Q(s+1) + \Xi}{2^{Q(s+1) - 1}}
\\ & \le 1.
\end{split}
\end{equation}
In view of part~(b) of Lemma~\ref{l:the itinerary}, formulas~\eqref{eq:N J} and~\eqref{eq:B J}, the first inequality of~\eqref{eq:geometrica1}, the first equality in~\eqref{eq:10}, and the hypothesis~$q + \Xi \ge 1$, we deduce that for~$j = s_0$ we have
\begin{equation}
\begin{split}
& \frac{1}{8} 2^{(1-\tau )Q(s_0+1)-\tau (\Xi - 2 \xi)(s_0+1)}
\\ & \le
\frac{1}{4} |J_{s_0}| 2^{- \tau (Q(s_0+1) + \Xi \cdot (s_0+1)) + 2 \tau  \xi \cdot (s_0 + 1)}
\\ & \le
\frac{1}{2} |J_{s_0}| 2^{- \lambda(s) (2^{Q(s_0)}+Q(s_0+1)-Q(s_0) + \Xi - 1) -\tau(Q(s_0+1) + \Xi \cdot (s_0+1)) + 2 \tau \xi \cdot (s_0 + 1)}
\\ & \le
\left(\sum_{m = 1}^{|J_{s_0}|} 2^{- \lambda(s) m} \right) 2^{- \lambda(s) (2^{Q(s_0)}+Q(s_0+1)-Q(s_0) + \Xi - 1) -\tau(Q(s_0+1) + \Xi \cdot (s_0+1)) + 2 \tau \xi \cdot (s_0 + 1)}
\\ & =
J_{s_0}(\tau, \lambda(s)).
\end{split}
\end{equation}
This proves part~$3$ of the lemma and completes the proof.
\endproof

Define the function~$s : (- \infty, 1) \to \R$ by
$$ s(\tau) = \frac{\Xi - 2\xi}{q(1 - \tau)}. $$
\begin{lemm}\label{lem:22}
For every~$\tau$ in~$\left( \tfrac{q-2}{q-1}, 1 \right)$, we have
\[
 \Pi \left( \tau , \lambda \left( s(\tau) - 2 \right) \right)
\le
10 \cdot 2^{\tau \xi} + 101,
\]
and for every~$\Omega \ge 0$, we have
\[
2^{\Omega - 3}
\le
\Pi \left(\tau , \lambda \left(s(\tau) + \Omega \right) \right).
\]
\end{lemm}
\proof
Fix~$\tau$ in~$\left( \tfrac{q-2}{q-1}, 1 \right)$ and~$\Omega \ge 0$.
Note that the inequality~$\tau > \tfrac{q-2}{q-1}$ implies~$\tau > \tfrac{q-4}{q}$.
Moreover, this last inequality is equivalent to
$$ \tau
>
\frac{Q\left(\frac{2}{q(1-\tau)} - 1\right) - 1}{Q\left(\frac{2}{q(1-\tau)}\right)}.$$
On the other hand, the function~$s \mapsto \frac{Q(s - 1) - 1}{Q(s)}$ is strictly increasing on~$\left( \tfrac{q-1}{q}, + \infty \right)$.
Since the inequalities~$1\le \Xi -2\xi$ and~$\tau > \tfrac{q-2}{q-1}$ imply $s(\tau) > \tfrac{q-1}{q}$, using $\Xi -2\xi\le 2$ we deduce 
$$
\frac{Q(s(\tau) - 1) - 1}{Q(s(\tau))}
\le
\frac{Q\left(\frac{2}{q(1-\tau)} - 1\right) - 1}{Q\left(\frac{2}{q(1-\tau)}\right)}
<
\tau.$$
So the hypotheses of part~$2$ of Lemma~\ref{lem:11} are satisfied with~$s =
s(\tau) - 2$.
Let~$F : \R \to \R$ be the quadratic function defined by
$$ F(\ell) \= (1 - \tau) Q(\ell) - \tau (\Xi - 2\xi) \ell. $$
Note that $F(0) = 0$,
$$ F \left( \frac{s(\tau)}{2} \right)
=
\frac{(\Xi - 2\xi)^2}{2q} - \frac{\Xi - 2\xi}{2} \left( \frac{s(\tau)}{2}
\right)
\text{ and }
F(s(\tau)) = \frac{(\Xi - 2\xi)^2}{q}. $$
Using that~$F$ is convex, we conclude that for each~$\ell$ in~$[0, s(\tau)]$ we
have
\begin{multline*}
F(\ell) = (1 - \tau) Q(\ell) - \tau (\Xi - 2 \xi) \ell
\\ \le
\frac{(\Xi - 2\xi)^2}{q} - \frac{\Xi - 2 \xi}{2} \min \left\{ \ell, s(\tau) - \ell \right\}.
\end{multline*}
Therefore, putting~$s^+ = s(\tau) - 2$ and using~$1 \le \Xi - 2\xi \le 2$ and~$q \ge 3$, we have
\begin{multline*}
\sum_{j=0}^{\lfloor s^+ \rfloor + 1} 2^{(1- \tau)Q(j+1) - \tau (\Xi - 2\xi) (j+1)}
\le
2 \sum_{\ell = 0}^{\lfloor s^+ \rfloor + 2 } 2^{\frac{(\Xi - 2\xi)^2}{q}
-\frac{\Xi - 2\xi}{2} \ell}
\\ \le
2 \cdot 2^{\frac{4}{q}} \frac{1}{1-2^{-\frac{1}{2}}}
\le
20.
\end{multline*}
The first inequality of the lemma is then obtained using part~$2$ of
Lemma~\ref{lem:11} with~$s = s^+$.

To prove the second inequality, note that
\begin{equation*}
F(s(\tau) + \Omega)
= \frac{(\Xi-2\xi)^2}{q}+
(\Xi - 2\xi) (2 - \tau) \Omega + q(1 - \tau) \Omega^2
\ge
(\Xi - 2\xi) \Omega
\ge
\Omega
\end{equation*}
and that~$F$ is increasing on the interval~$\left[\frac{\tau}{2}  \frac{\Xi - 2\xi}{q(1 -
\tau)}, + \infty \right)$, that contains~$s(\tau)$.
So, if we put~$s^- = s(\tau) + \Omega$, then
\begin{equation*}
\Omega
\le
F(s^-)
\le
F({\lfloor s^- \rfloor} + 1)
=
(1 - \tau) Q({\lfloor s^- \rfloor} + 1) - \tau (\Xi - 2\xi) ({\lfloor s^-
\rfloor} + 1).
\end{equation*}
Together with part~$3$ of Lemma~\ref{lem:11} with~$s = s^-$, we obtain
\[
2^{\Omega}
\le
2^{(1 -\tau )Q(\lfloor s^- \rfloor + 1) - \tau (\Xi - 2\xi)(\lfloor s^- \rfloor
+ 1)}
\le
8 \Pi \left(\tau , \lambda \left(s^- \right) \right),
\]
from which we obtain the second inequality of the lemma.
\endproof
\begin{proof}[Proof of Proposition~\ref{p:estimating 2 variables series}]
The first inequality is part~$1$ of Lemma~\ref{lem:11}.
To prove the other inequalities, note that by the definition of~$\lambda(s)$ we
have~$\lambda(s) \ge 2^{-Q(s + 1)}$.
On the other hand, by part~(b) of Lemma~\ref{l:the itinerary} we have~$\lambda(s) \le 2 \cdot 2^{- Q(s +
1)}$.
So, using the definition of the function~$s$ we have for each~$\tau$ in~$(0, 1)$
and~$\Delta \ge  1$,
$$ \lambda(s(\tau) - 2)
\le
2 \cdot 2^{- q \left( \frac{\Xi - 2\xi}{q(1 - \tau)} - 1 \right)^2} $$
and
$$ \lambda(s(\tau) + \Delta - 1)
\ge
2^{- q \left( \frac{\Xi - 2\xi}{q(1 - \tau)} + \Delta\right)^2}. $$
Then the desired inequalities are a direct consequence of Lemma~\ref{lem:22}
with~$\Omega = \Delta -1$ and of the fact that for a fixed~$\tau$ the function
$$ \lambda \mapsto \Pi(\tau, \lambda)$$
is nonincreasing on the set where it is finite.
\end{proof}

\section{Estimating the geometric pressure function}
\label{s:finale}

In this section we prove the Main Technical Theorem.
In~\S\ref{ss:conformal measures} we show a general result about conformal 
measures, and in~\S\ref{ss:transition parameter} we make some technical estimates
(Proposition~\ref{p:transition divergence}).
The proof of Main Technical Theorem is 
in~\S\ref{ss:proof of Main Technical Theorem}, after recalling a few results
from~\cite{CorRiva}.
\subsection{Conformal measures}
\label{ss:conformal measures}
Recall that, given an integer~$n \ge 3$ and a parameter~$c$ in~$\cK_n$, the \emph{conical} or \emph{radial Julia set of~$f_c|_{I_c}$} (resp.~$f_c$) is the set of all points~$x$ in~$I_c$ (resp.~$J_c$) for which the following property holds: There exists~$r > 0$ and an unbounded sequence of positive integers~$(n_j)_{j = 1}^{+ \infty}$, such that for every~$j$ the map~$f_c|_{I_c}^{n_j}$ (resp.~$f_c^{n_j}$) maps a neighborhood of~$x$ in~$I_c$ (resp.~$J_c$) diffeomorphically to~$B(f_c^{n_j}(x), r)$.

\begin{prop}
\label{p:conformal measures}
Let~$n \ge 4$ be an integer, $c$ a parameter in~$\cK_n$, and let~$t > 0$ and~$p$ in~$\R$ be given.
Then there is at most one $(t, p)$\nobreakdash-conformal probability measure of~$f_c|_{I_c}$ (resp.~$f_c$) supported on~$I_c$ (resp.~$J_c$).
If such a measure~$\mu$ exists, then~$p \ge P_c^{\R}(t)$ 
(resp. $p \ge P_c^{\C}(t)$),
and~$\mu$ is either supported on the backward orbit of~$0$ and dissipative, or~$\mu$ is nonatomic
and supported on the conical Julia set of~$f_c|_{I_c}$ (resp.~$f_c$).
Furthermore, the former case holds precisely when following series converges:
\begin{multline}
  \label{e:Poincare series}
 \sum_{j = 1}^{+ \infty} \exp(- j p) \sum_{y \in f_c|_{I_c}^{-j}(0)} |Df_c^j (y)|^{-t}  
\\ 
\left( \text{resp. } \sum_{j = 1}^{+ \infty} \exp(- j p) \sum_{y \in f_c^{-j}(0)} |Df_c^j (y)|^{-t} \right).
\end{multline}
\end{prop}
\begin{proof}
By~\cite[Theorem~$4.2$]{Urb03b} the conical Julia set of~$f_c$ is the complement
 in~$J_c$ of the backward orbit of~$z = 0$.
This implies that the conical Julia set of~$f_c|_{I_c}$ contains the complement 
in~$I_c$ of the backward orbit of~$z = 0$ under~$f_c|_{I_c}$.
On the other hand, this last set is clearly disjoint from the conical Julia set 
of~$f_c|_{I_c}$, so this proves that the conical Julia set of~$f_c|_{I_c}$ is the
complement in~$I_c$ of the backward orbit of~$z = 0$.

Let~$\mu$ be a~$(t, p)$\nobreakdash-conformal probability measure
for~$f_c|_{I_c}$ (resp.~$f_c$) supported on~$I_c$ (resp.~$J_c$).
If~$\mu$ is supported on the backward orbit of~$z = 0$, then it is uniquely determined by the mass it assigns to~$z = 0$, and therefore it is unique up to a scalar factor.
Note moreover that in this case~$\mu$ is dissipative, because it charges the wandering set~$\{ 0 \}$.
If~$\mu$ is not entirely supported on the backward orbit of~$z = 0$, then it charges the conical Julia set, so~$\mu$ is nonatomic, it is supported on the conical Julia set and it is the unique $(t, p)$\nobreakdash-conformal measure of~$f_c|_{I_c}$ (resp.~$f_c$) supported on~$I_c$ (resp.~$J_c$), up to a scalar factor, see~\cite[Proposition~$4.1$]{PrzRiv11} for the complex case; the proof of the uniqueness part of this result applies without change to the real case.
This completes the proof that~$\mu$ is unique.

To prove that in the complex case we have~$p \ge P_c^{\C}(t)$, let~$\delta > 0$
be sufficiently small so that~$B(0, 2 \delta)$ is disjoint from the forward
orbit of the critical point.
It follows that there is a constant~$K > 1$ such that for every integer~$j \ge 1$ and every~$y$ in~$f_c^{-j}(0)$, the map~$f_c^j$ maps a neighborhood~$W_y$ of~$y$ biholomorphically to~$B(0, \delta)$ with distortion bounded by~$K$.
Therefore,
$$ \mu(W_y)
\ge
K^{-t} \exp(- j p) |Df_c^j (y)|^{-t} \mu(B(0, \delta)). $$
So, if we put~$C \= K^{-1} \mu(B(0, \delta)) > 0$, then for every integer~$j \ge 1$ we have
$$ 1
\ge
\sum_{y \in f_c^{-j}(0)} \mu(W_y)
\ge
C \exp(- j p) \sum_{y \in f_c^{-j}(0)} |Df_c^j (y)|^{-t}. $$
Since by~\cite[Theorem~A]{PrzRivSmi04} we have
$$ \lim_{j \to + \infty} \frac{1}{j} \log \sum_{y \in f_c^{-j}(0)} |Df_c^j (y)|^{-t}
=
P_c^{\C}(t), $$
this proves~$p \ge P_c^{\C}(t)$.

To prove that in the real case we have~$p \ge P_c^{\R}(t)$, we note that the proof of~\cite[Proposition~$2.1$]{PrzRivSmi04} can be adapted to show
$$ \lim_{j \to + \infty} \frac{1}{j} \log \sum_{y \in (f_c|_{I_c})^{-j}(0)} |Df_c^j (y)|^{-t}
=
P_c^{\R}(t), $$ using the fact that~$z = 0$ is not in the closure of the orbit of the critical value of~$f_c$.
The rest of the proof of~$p \ge P_c^{\R}(t)$ is similar to the proof above.

To prove the last statement, observe first that if there is a $(t, p)$\nobreakdash-conformal measure for~$f_c|_{I_c}$ (resp.~$f_c$) that is supported on the backward orbit of~$z = 0$, then its total mass is equal to~\eqref{e:Poincare series} times the mass at~$z = 0$.
This proves that~\eqref{e:Poincare series} is finite.
Conversely, if~\eqref{e:Poincare series} is finite, then
\begin{multline*}
\delta_0 + \sum_{j = 1}^{+ \infty} \sum_{y \in f_c|_{I_c}^{-j}(0)} \exp(- j p) |Df_c^j (y)|^{-t} \delta_y
\\
\left( \text{ resp. $\delta_0 + \sum_{j = 1}^{+ \infty} \sum_{y \in f_c^{-j}(0)} \exp(- j p) |Df_c^j (y)|^{-t} \delta_y $} \right).
\end{multline*}
is finite and it is a $(t, p)$\nobreakdash-conformal measure for~$f_c|_{I_c}$ (resp.~$f_c$) supported on~$I_c$ (resp.~$J_c$).
\end{proof}

\subsection{Phase transition parameter}
\label{ss:transition parameter}
Recall that for each parameter~$c$ in~$\cP_5(-2)$, we have put
\[
\theta(c)
=
\left| \frac{Dg_c(p(c))}{Dg_c(\wtp(c))} \right|^{1/2}
\text{ and }
\xi(c)
=
- \frac{ \log (\zeta(c)\tzeta(c))}{4\log \theta(c)}.
\]
Put~$t(c) \= \frac{\log 2}{\log \theta(c)}$ and for every integer $n\ge 5$ put
\[
\xi_{n}
\=
\sup_{c\in \cK_{n}} \xi(c).
\]

This subsection is dedicated to prove the following estimates, used in the proof of the Main Technical Theorem.
\begin{prop}
\label{p:transition divergence}
There is an integer $n_2\ge 5$ such that for every integer $n\ge n_2$, and every~$c$ in~$\cK_{n}$, we have
$\lceil 2\xi_{n} + 1 \rceil - 2\xi(c)  \le 2$.
Furthermore, for every constant~$T > 0$ there is~$n_3 \ge 5$ such that for every integer~$n \ge n_3$ and every parameter~$c$ in~$\cK_n$, we have~$t(c) \ge T$.
\end{prop}
The proof of this proposition is at the end of this subsection and it follows
from the following sequence of lemmas.  
\begin{lemm}[\cite{CorRiva}, Lemma~A.$1$]
\label{l:period 3 transversality}
We have,
$$ \left. \frac{\partial}{\partial c} |Df_c^3(p(c))| \right|_{c = - 2}
>
\left. \frac{\partial}{\partial c} |Df_c^3(\wtp(c))| \right|_{c = - 2}. $$
\end{lemm}

\begin{lemm}
 \label{l:Derivative of theta at -2}
We have~$\theta(-2) = 1$ and $D\theta(-2)>0$.
\end{lemm}
\proof
For~$c = -2$,
$$ \{ 2 \cos (2\pi/7), 2 \cos (4 \pi/ 7), 2 \cos(6 \pi/7) \}$$
and
$$\{ 2 \cos (2\pi/9), 2 \cos (4 \pi/ 9), 2 \cos(8 \pi/9) \} $$
are the only orbits of minimal period~$3$ of $f_{-2}$.
Thus, we have
$$ |Df_{-2}^3(p(-2))| = |Df_{-2}^3(\wtp(-2))| = 8,
\text{ and }
\theta(-2) = 1. $$
Together with Lemma~\ref{l:period 3 transversality}, we obtain $D\theta(-2)>0$.
\endproof

\begin{lemm}\label{l:analyticity of xi}
We have $\zeta(-2)\cdot \tzeta(-2) = 1$, and the function~$\xi(c)$ is real analytic at $c = -2$.
\end{lemm}
\proof
Let~$c$ be a parameter in $\cP_5(-2)$.
For every integer $m\ge 1$ denote by~$p_m(c)$ the periodic point in~$\Lambda_c$ whose itinerary consists of the periodic sequence whose period is the concatenation of~$m$ consecutive~$0$'s and of~$m$ consecutive~$1$'s.
By Proposition~\ref{p:Improved distortion estimate}, 
$$
\frac{Dg_c^{2m}(p_m(c))}{(Dg_c(p(c))Dg_c(\widetilde p(c)))^m} \to \zeta(c)\tzeta(c)  \quad \text{ as } m\to +\infty. 
$$
On the other hand, using the identity $f_{-2}(2\cos(x)) = 2 \cos(2x)$ for~$x$ in~$\R$, we obtain
$$
\frac{Dg_{-2}^{2m}(p_m(-2))}{(Dg_{-2}(p(-2))Dg_{-2}(\widetilde p(-2)))^m} = 1.
$$
This proves~$\zeta(-2) \cdot \tzeta(-2) = 1$.
To prove that~$\xi$ is real analytic at~$c = - 2$, notice that each of the functions~$\theta$, $\zeta$, and~$\tzeta$ is real analytic at $c= -2$  (Proposition~\ref{p:Improved distortion estimate}).
Since $\zeta(-2)\cdot \tzeta(-2) = 1$ and   $\theta(-2) = 1$ (Lemma~\ref{l:Derivative of theta at -2}), we have that the functions~$A$ and~$B$ defined for~$c$ in~$\cP_5(-2)$ by
$$ A(c)\= \log (\zeta(c) \tzeta(c)) / (c + 2)
\text{ and }
B(c) \= \log \theta(c) / (c + 2), $$
are also real analytic at $c = -2$.
Moreover, $B(-2) \neq 0$ since by Lemma~\ref{l:Derivative of theta at -2} we have $D\theta(-2) \neq 0$.
Thus, the quotient~$4\xi(c) = A(c)/B(c)$ is real analytic at $c = -2$.
This concludes the proof of the lemma.
\endproof

\proof[Proof of Proposition~\ref{p:transition divergence}]
By Lemma~\ref{l:Derivative of theta at -2}, there is $\delta>0$ such that for
every~$c$ in $(-2, -2 + \delta)$ we have
$$ 1 < \theta(c) < 2^{1/T}. $$ 
On the other hand, by Proposition~\ref{p:ps}, there is an integer $n_0\ge 3$ such that for every $n\ge n_0$ we have
$\cK_n \subset (-2, -2 + \delta)$.
These assertions imply the second part of the proposition.

To prove the first part, notice that by Lemma~\ref{l:analyticity of xi} there is $\epsilon > 0$ such that~$\xi(c)$ is uniformly continuous on the interval $[-2, -2 + \epsilon]$.
By Proposition~\ref{p:ps}, for every sufficiently large integer~$n$ we have $\cK_n \subset [-2, -2 + \epsilon]$ and moreover the diameter of~$\cK_{n}$ converges to~$0$ as $n \to +\infty$.
Thus, for every sufficiently large~$n$ we have
$$ \xi_{n} - \xi(c) < \frac{1}{2}. $$
This implies the first assertion of the proposition and concludes the proof.
\endproof

\subsection{Proof of the Main Technical Theorem}
\label{ss:proof of Main Technical Theorem}
We start recalling some results from~\cite{CorRiva}.
\begin{prop}[\cite{CorRiva}, Proposition~D]
\label{p:improved MS criterion}
There is an integer~$n_5 \ge 4$ and a constant~$C_3 > 1$, such that for every integer~$n \ge n_5$ and every parameter~$c$ in~$\cK_n$, the following properties hold for each~$t \ge 3$:
\begin{enumerate}
\item[1.]
For~$p$ in~$[- t \chicrit/2, 0)$ satisfying
$$ \sum_{k = 0}^{+ \infty} \exp(- (n + 3k)p)|Df_c^{n + 3k} (c)|^{-t/2}
\ge
C_3^{t}, $$
we have~$\sP_c^{\R}(t, p) > 0$ and~$P_c^{\R}(t) \ge p$.
\item[2.]
For~$p \ge - t \chicrit/2$ satisfying
$$ \sum_{k = 0}^{+ \infty} \exp(- (n + 3k)p)|Df_c^{n + 3k} (c)|^{-t/2}
\le
C_3^{-t}, $$
we have~$\sP_c^{\C}(t, p) < 0$ and~$P_c^{\C}(t) \le p$.
\end{enumerate}
\end{prop}
\begin{lemm}[\cite{CorRiva}, Proposition~$6.2$]
\label{l:critical line}
For every integer $n \ge 5$, every parameter~$c$ in~$\cK_n$, and every~$t > 0$, we have
$$ P_c^{\C}(t) \ge P_c^{\R}(t) \ge - t \chicrit / 2. $$
\end{lemm}
\begin{lemm}[\cite{CorRiva}, Lemma~$6.5$]
\label{l:Poincare series}
There is~$n_6 \ge 5$, such that for every integer~$n \ge n_6$, every~$c$ in~$\cK_n$, and every~$t \ge 3$ and
$$ p \ge P_c^{\R}(t)
\left( \text{resp. } p \ge P_c^{\C}(t) \right) $$
satisfying~$\sP_c^{\R}(t, p) < 0$ (resp. $\sP_c^{\C}(t, p) < 0$), the
sum~\eqref{e:Poincare series} is finite.
\end{lemm}
\begin{proof}[Proof of the Main Technical Theorem]
Let~$C_2$ and~$\upsilon_2$ be the constants given by Proposition~\ref{p:2 variables series}, $n_5$ and $C_3$ the constants given by Proposition~\ref{p:improved MS criterion}, and~$n_6$ the constant given by Lemma~\ref{l:Poincare series}.
Since for~$c = - 2$ we have
$$ |Dg_{-2} (\wtp(-2))|^{1/3} = 2
\text{ and }
|Df_{-2} (\beta(-2))| = 4, $$
there is~$\delta > 0$ such that that for each~$c$
in~$(-2, -2 + \delta)$ we have
\begin{equation}
\label{e:first entrance derivative}
\frac{|Dg_c(\wtp(c))|^{1/3}}{|f_c(\beta(c))|}
<
\frac{2}{3}.
\end{equation}
By Proposition~\ref{p:ps} there is $n_0 \ge 3$ such that for every integer~$n\ge n_0$ the set~$\cK_n$ is contained in~$(-2, -2 + \delta)$; thus for every~$c$ in~$\cK_n$ we have~\eqref{e:first entrance derivative}.
Since the closure of~$\cP_6(-2)$ is contained in~$\cP_5(-2)$ (part~$1$ of Lemma~\ref{l:auxiliary para-puzzle pieces}), by Proposition~\ref{p:Improved distortion estimate} we have,
\[
Z
\=
\sup_{c\in \cP_6(-2)} - \frac{\log (\zeta(c)\tzeta(c)) }{ 4\log 2}
<
+ \infty.
\]
Fix~$n \ge \max \{ 6,  n_0, n_5, n_6 \}$ large enough such that
\begin{equation}
\label{e:bound with Z}
C_2 \left( \tfrac{2}{3} \right)^{\frac{1}{2}n}
\left(10\cdot 2^{Z}  + 101\right)
< C_3^{-1}.
\end{equation}
In view of Proposition~\ref{p:transition divergence}, we can take~$n$ larger if 
necessary so that for every~$c$ in~$\cK_n$ we have
$$ t(c) = \frac{\log 2}{\log \theta(c)} \ge 6 $$
and such that, if we put
\[
\Xi \= \left\lceil 2\xi_{n} + 1\right\rceil,
\]
then for every~$c$ in~$\cK_{n}$ we have~$\Xi - 2 \xi(c) \le 2$.
Consider the sequence~$(x_j)_{j = 0}^{+ \infty}$ in~$\{0, 1\}^{\N_0}$ defined
in~\S\ref{ss:the itinerary} for this value of~$\Xi$ and for some 
integer~$q \ge 3$ satisfying in addition,
$$ q + \Xi \ge 1 \text{ and }~2^{q - 1} \ge q + 1 + \Xi.$$
By Proposition~\ref{p:ps}, there is a parameter~$c$ in~$\cK_{n}$ such that
$\iota(c) = (x_j)_{j = 0}^{+ \infty}$.
Finally, put~$t_* \= t(c)$, and fix $\Delta \ge  1$ sufficiently large such that
\begin{equation}
\label{eq:bound Lambda from below}
C_2^{-t_*} \exp(-n )\left( \frac{\exp(\chicrit)}{|Df_c(\beta(c))|}
\right)^{\frac{t_*}{2} n } ~2^{\Delta - 4}
>
C_3^{t_*}.
\end{equation}
Put
$$ t_0 \= \frac{q - 2}{q - 1} t_*,
\xi \= \xi(c),
\text{ and }
\kappa \= \Xi - 2 \xi, $$
and define the functions
$$ \delta^+, \delta^-, p^+, p^- : (t_0, + \infty) \to \R $$
as in the statement of the Main Technical Theorem.
Taking~$\Delta$ larger if necessary, assume that for every~$t$ in~$(t_0, + \infty)$ we have~$p^-(t) < 0$.

We start showing that~$c$ satisfies the hypotheses of Proposition~\ref{p:2 variables series}.
By~\eqref{eq:N J} and~\eqref{eq:N I}, we have
\begin{equation}
  \label{e:mostly 1's}
\frac{N_c(k)}{k} \to 0 \text{ as } k\to +\infty.
\end{equation}
Denote by~$(m_j)_{j = 0}^{+ \infty}$ the sequence of lengths of the blocks of~$0$'s and~$1$'s in~$\iota(c)$.
So, using the notation in~\S\ref{ss:the itinerary}, for every integer~$s \ge 0$ we have~$m_{2s} = |I_s|$ and~$m_{2s + 1} = |J_s|$.
By part~(a) of Lemma~\ref{l:the itinerary} we have for every integer~$s \ge 0$
$$ \min \{ m_{2s}, m_{2s + 1} \}
=
\min \{ |I_s|, |J_s| \}
=
|I_s|
=
q(2s + 1) + \Xi, $$
and for every integer~$s \ge 1$
$$ \min \{ m_{2s + 1}, m_{2s + 2} \}
=
\min \{ |J_s|, |I_{s + 1}| \}
=
|I_{s + 1}|
=
q(2s + 3) + \Xi. $$
Thus
\begin{equation*}
 \sum_{j=2}^{+\infty} \exp(-\min\{m_j,m_{j+1}\} \upsilon_2)
\\ \le
2 \sum_{s=1}^{+\infty} \exp(- (q(2s + 1) + \Xi) \upsilon_2)
< 
+ \infty.
\end{equation*}
This proves~$c$ satisfies the hypotheses of Proposition~\ref{p:2 variables series}.

Note that by our choice of~$n$ and the hypothesis~$q \ge 3$, we have
$$ t_0 = \frac{q - 2}{q - 1} t_* \ge \frac{1}{2} t_* \ge 3. $$
On the other hand, by~\eqref{e:mostly 1's} and Lemma~\ref{lem:chicrit} we have
\begin{equation}
\label{e:Lyapunov formula}
\exp(\chicrit) = |Dg_c(\wtp(c))|^{1/3}.
\end{equation}
In particular, $\chicrit > 0$.
Consider the~$2$ variables series~$\Pi$ defined as in~\S\ref{ss:estimating 2 variables series} for the above choices of~$\Xi$, $q$, and~$\xi$, and note that it coincides with the~$2$ variables series~$\Pi_c$ defined in~\S\ref{ss:2 variables series} for our choice of the parameter~$c$.

To prove that for every~$t > t_0$ we have~$P_c^{\R}(t) \ge p^-(t)$, note first that when~$t \ge t_*$ this is given by Lemma~\ref{l:critical line}.
On the other hand, from Propositions~\ref{p:2 variables series} and~\ref{p:estimating 2 variables series}, \eqref{e:first entrance derivative}, \eqref{eq:bound Lambda from below}, \eqref{e:Lyapunov formula}, and from the fact that for every~$t$ in~$(t_0, t_*)$ we have~$\delta_-(t) \le 1$, we deduce
\begin{equation*}
\begin{split}
~& \sum_{k=0}^{+\infty} \exp(- (n+3k) p^-(t))
|Df_c^{n+3k}(c)|^{-\frac{t}{2}}\
\\ & =
\sum_{k=0}^{+\infty} \exp \left( -(n+3k) \left( - t \frac{\chicrit}{2} + \delta^-(t) \right) \right)
|Df_c^{n+3k}(c)|^{-\frac{t}{2}}
\\ & \ge
C_2^{-t} \exp(-n\delta^-(t))\left( \frac{\exp(\chicrit)}{|Df_c(\beta(c))|} \right)^{\frac{t}{2}n } ~\Pi\left(t  \frac{\log \theta(c)}{\log 2}, \frac{3\delta^-(t)}{\log 2} \right)
\\ & \ge
C_2^{-t} \exp(-n\delta^-(t))\left( \frac{\exp(\chicrit)}{|Df_c(\beta(c))|} \right )^{\frac{t}{2} n } ~2^{\Delta - 4}
\\ & >
C_3^{t_*}
\\ & \ge
C_3^{t}.
\end{split}
\end{equation*}
Since for each~$t$ in~$(t_0, t_*)$ we have~$p^-(t) < 0$, the inequality above
combined with part~$1$ of Proposition~\ref{p:improved MS criterion} implies that
for every~$t$ in $(t_0,t_*)$ 
we have~$P_c^{\R}(t) \ge p^-(t)$.

Now we turn to the proof that for every~$t > t_0$ we have~$P_c^{\C}(t) \le
p^+(t)$ and that for every~$t \ge t_*$ we have~$\sP_c^{\C}(t, - t
\frac{\chicrit}{2}) < 0$.
Combining Propositions~\ref{p:2 variables series} and~\ref{p:estimating 2 variables series}, using the definition of~$\xi = \xi(c)$ and~$Z$, and using~\eqref{e:first entrance derivative}, \eqref{e:bound with Z} and~\eqref{e:Lyapunov formula}, we deduce that for every $t \ge t_*$
\begin{equation}\label{e:bound series}
\begin{split}
& \sum_{k=0}^{+\infty} \exp(-(n+3k)p^+(t))
|Df_c^{n+3k}(c)|^{-\frac{t}{2}}
\\ & =
\sum_{k=0}^{+\infty} \exp\left(- \left(n+3k\right)\left(- t \frac{\chicrit}{2} \right)\right)
|Df_c^{n+3k}(c)|^{-\frac{t}{2}}
\\ & \le
C_2^t \left( \frac{\exp(\chicrit)}{|Df_c(\beta(c))|} \right)^{\frac{t}{2} n }\Pi\left(t  \frac{\log \theta(c)}{\log 2}, 0\right)
\\ & \le
C_2^t \left( \frac{\exp(\chicrit)}{|Df_c(\beta(c))|} \right)^{\frac{t}{2} n } 2 \left(2^{ t  \frac{\log \theta(c)}{\log 2} \xi} + 1 \right)
\\ & \le
C_2^t\left( \frac{\exp(\chicrit)}{|Df_c(\beta(c))|} \right)^{\frac{t}{2} n} 2 \left( 2^{tZ}  + 1 \right)
\\ & <
C_3^{-t}
\end{split}
\end{equation}
and that for every~$t$ in~$(t_0, t_*)$ we have
\begin{equation*}
\begin{split}
~& \sum_{k=0}^{+\infty} \exp(- (n+3k)p^+(t))
|Df_c^{n+3k}(c)|^{-\frac{t}{2}}
\\ & =
\sum_{k=0}^{+\infty} \exp \left(-(n+3k) \left( - t \frac{\chicrit}{2} + \delta^+(t)\right) \right)
|Df_c^{n+3k}(c)|^{-\frac{t}{2}}
\\ & \le
C_2^t \exp(-n\delta^+(t))\left( \frac{\exp(\chicrit)}{|Df_c(\beta(c))|}
\right
)^{\frac{t}{2}
n } ~\Pi \left(t  \frac{\log \theta(c)}{\log 2}, \frac{3\delta^+(t)}{\log
2} \right)
\\ & \le
C_2^t \exp(-n\delta^+(t))\left( \frac{\exp(\chicrit)}{|Df_c(\beta(c))|} \right)^{\frac{t}{2} n} \left( 10 \cdot 2^{t  \frac{\log \theta(c)}{\log 2} \xi} + 101 \right)
\\ & \le
C_2^t \exp(-n\delta^+(t))\left( \frac{\exp(\chicrit)}{|Df_c(\beta(c))|}
\right
)^{\frac{t}{2} n} \left( 10 \cdot 2^{t Z} + 101 \right)
\\ & <
C_3^{-t}.
\end{split}
\end{equation*}
Since for~$t > t_0$ we have $p^+(t) \ge -\frac{t}{2}\chicrit$, applying part~$2$
of Proposition~\ref{p:improved MS criterion} we deduce that for~$t > t_0$ 
we have~$P_c^{\C}(t) \le p^+(t)$ and~$\sP_c^{\C}(t, - t \frac{\chicrit}{2}) < 0$.

To prove the assertions concerning conformal measures, recall that we have proved that for~$t \ge t_*$ we have
$$ P_c^{\R}(t) = P_c^{\C}(t) = - t \chicrit/2 $$
and~$\sP_c^{\C} \left( t, - t \frac{\chicrit}{2} \right) < 0$.
This implies that for~$p \ge - t \chicrit/2$ we have
$$ \sP_c^{\R}(t, p)
\le
\sP_c^{\C} \left( t, p \right)
\le
\sP_c^{\C} \left( t, - t \frac{\chicrit}{2} \right) < 0. $$
So the assertions about conformal measures follow from Proposition~\ref{p:conformal measures} and Lemma~\ref{l:Poincare series}.

To prove the assertions about equilibrium states, let~$t \ge t_*$ be given and suppose by contradiction there is an equilibrium state~$\rho$ of~$f_c|_{I_c}$ (resp.~$f_c$) for the potential~$- t \log |Df_c|$.
Since~$f_c$ satisfies the Collet-Eckmann condition, it follows that the Lyapunov exponent of~$\rho$ is strictly positive, see~\cite[Theorem~A]{NowSan98} or~\cite[Main Theorem]{Riv1204} for the real case and~\cite[Main Theorem]{PrzRivSmi03} for the complex case.
Then~\cite[Theorem~$6$]{Dob1304} in the real case and~\cite[Theorem~$8$]{Dob12} in the complex case imply that~$\rho$ is absolutely continuous with respect to the $(t, - t \chicrit/2)$\nobreakdash-conformal measure for~$f_c|_{I_c}$ (resp.~$f_c$) that is supported on~$I_c$ (resp.~$J_c$).
This implies in particular that~$\rho$ is supported on the backward orbit of~$z = 0$ and hence that~$\rho$ charges~$z = 0$.
This is impossible because this point is not periodic.
This contradiction shows that there is no equilibrium state of~$f_c|_{I_c}$ (resp.~$f_c$) for the potential~$- t \log |Df_c|$ and completes the proof of the Main Technical Theorem.
\end{proof}

\bibliographystyle{alpha}

\end{document}